\documentclass[pdflatex,sn-mathphys]{sn-jnl}

\usepackage{amssymb}
\usepackage{xcolor}
\usepackage{fancyhdr}
\normalbaroutside 
\usepackage{tikz}
\usetikzlibrary{arrows.meta}
\usetikzlibrary{positioning}
\usetikzlibrary{shapes}
\usetikzlibrary{calc}
\usepackage{multirow}
\usepackage{makecell}
\usepackage[all,pdf]{xy}
\usepackage{marvosym}
\usepackage{ifsym}
\usepackage{relsize}

\numberwithin{equation}{section}

\theoremstyle{thmstyleone}%
\newtheorem{theorem}{Theorem}[section]
\newtheorem{proposition}[theorem]{Proposition}%
\newtheorem{definition}[theorem]{Definition}
\newtheorem{lemma}[theorem]{Lemma}
\newtheorem{corollary}[theorem]{Corollary}
\newtheorem{example}[theorem]{Example}%
\newtheorem{remark}[theorem]{Remark}%

\newcommand{\RR}{\mathbb{R}}

\newcommand{\ZZ}{\mathbb{Z}}
\newcommand{\NN}{\mathbb{N}}
\newcommand{\CC}{\mathbb{C}}

\newcommand{\ii}{\mathbf{i}}

\newcommand{\abs}[1]{\left\lvert #1 \right\rvert}

\newcommand{\pt}{\partial}

\newcommand{\ind}{\mathrm{Ind}}
\newcommand{\SOne}{\mathbb{S}^1}
\newcommand{\STwo}{\mathbb{S}^2}
\newcommand{\SOthr}{\mathrm{SO}(3)}
\newcommand{\PSUtwo}{\mathrm{PSU}(2)}
\newcommand{\Uone}{\mathrm{U}(1)}
\newcommand{\Otwo}{\mathrm{O}(2)}

\newcommand{\Msph}[3]{\mathcal{M}sph_{#1,#2}(#3)}
\newcommand{\MsphC}[3]{\mathcal{M}sph^C_{#1,#2}(#3)}
\newcommand{\MsphD}[3]{\mathcal{M}sph^D_{#1,#2}(#3)}
\newcommand{\Msphgn}{\mathcal{M}sph_{g,n}(\vec{\theta})}

\newcommand{\MsphDF}[3]{\mathcal{M}sphF_{#1,#2}^D(#3)}
\newcommand{\MsphDT}[4]{\mathcal{M}sph_{#1,#2}^D(#3;#4)}

\newcommand{\MsphgnD}{\mathcal{M}sph^D_{g,n}(\vec{\theta})}
\newcommand{\MsphgnDF}{\mathcal{M}sphF_{g,n}^D(\vec{\theta})}
\newcommand{\MsphgnDT}{\mathcal{M}sph_{g,n}^D(\vec{\theta};\vec{T})}
\newcommand{\MsphgnDFT}{\mathcal{M}sphF_{g,n}^D(\vec{\theta};\vec{T})}

\newcommand{\MsphCF}[3]{\mathcal{M}sphF_{#1,#2}^C(#3)}
\newcommand{\MsphCT}[4]{\mathcal{M}sph_{#1,#2}^C(#3;#4)}
\newcommand{\MsphCFT}[4]{\mathcal{M}sphF_{#1,#2}^C(#3;#4)}
\newcommand{\MsphgnC}{\mathcal{M}sph^C_{g,n}(\vec{\theta})}
\newcommand{\MsphgnCF}{\mathcal{M}sphF_{g,n}^C(\vec{\theta})}
\newcommand{\MsphgnCT}{\mathcal{M}sph_{g,n}^C(\vec{\theta};\vec{T})}
\newcommand{\MsphgnCFT}{\mathcal{M}sphF_{g,n}^C(\vec{\theta};\vec{T})}

\newcommand{\teich}{Teichm\"{u}ller}

\newcommand{\PH}{Poincar\'{e}-Hopf}
\newcommand{\mfoli}{measured foliation}

\newcommand{\qdf}{quadratic differential}
\newcommand{\nbhd}{neighborhood}
\newcommand{\homeo}{homeomorphism}

\newcommand{\csphsf}{cone spherical surface}
\newcommand{\csphmtr}{cone spherical metric}

\newcommand{\SYM}{\textcolor{black}}
\newcommand{\DEF}{\emph}
\newcommand{\EMP}{\textcolor{black}}

\begin{document}

\title[Article Title]{Moduli Space of Dihedral Spherical Surfaces and Measured Foliations}


\author{Sicheng Lu and Bin Xu}







\abstract{
{\it Cone spherical surfaces} are orientable Riemannian surfaces with constant curvature one and a finite set of conical singularities. A subset of these surfaces, referred to as {\it dihedral surfaces}, is characterized by their monodromy groups, which notably preserve a pair of antipodal points on the unit two-sphere within three-dimensional Euclidean space. On each dihedral surface, we define a pair of transverse measured foliations that, in turn, comprehensively characterize the original dihedral surface. Furthermore, we introduce a variety of geometric decompositions and deformations specific to dihedral surfaces. As a practical application, we ascertain the dimension of the moduli space for dihedral surfaces given specified cone angles and topological types. This dimension acts as an indicator of the independent geometric parameters that determine the isometric classes of these surfaces.
}



\maketitle


\pagestyle{plain}

\textcolor{black}{\tableofcontents}

\section{Introduction}\label{sec:intro}

The investigation of conformal constant curvature metrics with conical singularities on Riemann surfaces represents a natural extension of uniformization theory. Given a compact Riemann surface \(X\) of genus \(g\) with \(n\) distinct points \(P := \{p_1, \ldots, p_n\}\) and an \emph{angle vector} \(\vec{\theta} := (\theta_1, \ldots, \theta_n) \in \mathbb{R}_+^n\) where \(\theta_i \neq 1\), a central question arises: the existence of a singular conformal metric \(g\) of constant curvature \(K \in \{-1, 0, +1\}\) with a conical singularity of angle \(2\pi\theta_i\) at \(p_i\).
Expressed in complex coordinates, this quest implies the existence of a local coordinate \(z\), where \(g\) can be written as
\[ \mathrm{d}s = \frac{2\alpha \lvert z \rvert^{\alpha-1} \lvert \mathrm{d}z \rvert}{1 + K \lvert z \rvert^2}, \]
with \(\alpha = \theta_i\) at \(p_i\) and \(\alpha = 1\) outside \(P\). This line of inquiry has historical roots tracing back to the works of Picard and Poincar\'{e} \cite{Poin1898, Pic1905}. A surface equipped with such a metric and constant curvature is referred to as a \textbf{cone hyperbolic, flat}, or \textbf{spherical surface}.

The Gauss-Bonnet Formula provides a necessary condition for the existence of such metrics, stipulating that \(2 - 2g + \sum_{i=1}^n (\theta_i - 1)\) should share the same signature as \(K\). For \(K = -1, 0\), this condition is also sufficient, and the metric is unique up to scaling \cite{Heins62, McO88, Tro91}. However, in the case of \(K = +1\), the problem remains wide open, with only partial results available, as found in \cite{Tro89, Tro91, LuoT92, UYa00, Erm04, LinWang10, BDMM11, EGT14, EG15, Erm20sph, LSX21, dBP22}. A comprehensive survey \cite{Erm21arxiv} detailing the story and progress is also recommended. Neglecting the underlying complex structures of cone spherical surfaces, we are aware of the constraint on their cone angles. For \(g > 0\), the Gauss-Bonnet condition is also sufficient \cite{MonPa19}. When \(g = 0\), in addition to the Gauss-Bonnet condition, the complete angle constraint involves additional profound conditions, detailed in \cite{MonPa16, Li19, Erm20}.

Partial results on the moduli space of spherical surfaces also exist. For example, small deformations of such surfaces are studied in \cite{MzZhu20, MzZhu22}. Local properties of the Teichm\"uller space with small prescribed cone angles, using geometric analysis and elliptic operators, are examined in \cite{MzWs17}. A walls-and-chambers structure of the moduli space with a fixed topological type and variable cone angles is obtained in \cite{Tah22}, through a geometric decomposition of spherical surfaces.
Results on spherical tori with a single conical singularity are particularly rich. These surfaces can be studied using elliptic functions and modular forms \cite{LinWang10, ChLinW15, ChKLin17, LinW17}. The moduli space with different prescribed angles is explicitly described in \cite{EMP20} through a geometric approach.


\bigskip
There is a special class of \csphsf s with special monodromy, called \DEF{dihedral} surfaces (Definition \ref{def:dihedral}), which we better understand. They are further divided into two subclasses, known as \DEF{co-axial} and \DEF{strict dihedral} surfaces. These surfaces are closely associated with meromorphic abelian and quadratic differentials on the underlying Riemann surfaces, with real periods, as discovered by \cite{CWWX15, SCLX20}.
The sufficient and necessary conditions on the cone angles for these surfaces have been completely obtained in \cite{GT23}.
This paper continue the study of dihedral surfaces and their moduli space, focusing on the geometric aspect. Let $\MsphgnC$ and $\MsphgnD$ denote the moduli space of co-axial and strict dihedral surfaces with a fixed genus and prescribed cone angles (Definition \ref{def:moduli_D}, \ref{def:moduli_C}). Each dihedral surface induces a pair of transverse \mfoli s $(F,G)$ (Subsection \ref{sec:ll_foli}), which provides a crucial decomposition of the surface, leading to a geometric coordinate for the moduli space. Then we compute the maximum number of independent parameters that determine isometric classes of such surfaces, which corresponds to the dimension of the moduli space (Theorem \ref{thm:dim_count_D}, \ref{thm:dim_count_C}).
\begin{theorem}\label{thm:Main}
    Given an angle vector $\vec{\theta}=(\theta_1, \cdots, \theta_n) \in \RR^n_+$ with $\theta_i\neq1$, and $2g-2+n>0$. Assume that all the moduli spaces as follows are non-empty.
    \begin{enumerate}
        \item The maximal real dimension of $\MsphgnD$ is $2M_E+M_O+2g-2$, where $M_E$ is the number of integers in $\vec{\theta}$, $M_O$ is the maximal even number not exceeding the number of half-integers in $\vec{\theta}$.
        \item The real dimension of $\MsphgnC$ is $2M_E+2g-1$ if $(g,n-M_E)\neq(0,0)$, where $M_E$ is the number of integers in $\vec{\theta}$. And the real dimension of $\MsphC{0}{n}{\vec{\theta}}$ is $2n-3$ when $g=0, M_E=n$.
    \end{enumerate}
\end{theorem}
The constraints on the cone angle vector such that all the above moduli spaces are non-empty are quite complicated, whose full list for various cases could be found in Gendron-Tahar's work \cite{GT23}.
For most settings of $(g,n)$, the only constrain is the strengthened Gauss-Bonnet inequality \cite[Theorem 1.2]{GT23}.

The term ``maximal dimension'' suggests that the moduli space may have disjoint components of different dimensions (see Example \ref{eg:disjoint_comp}).
This leads to a refined moduli space by prescribing the topological type of the foliation $(F,G)$ (see Definition \ref{def:moduli_D}, \ref{def:moduli_C}). This is equivalent to restricting to a certain stratum of the differentials. In fact, we count the dimension of such refined moduli spaces (Theorem \ref{thm:dim_count_DT}, \ref{thm:dim_count_CT}).

Similar results for the case of $g=0$ are also obtained in \cite{CLY22axv} by solving curvature equations. It is important to note that the moduli space defined in this work is slightly different from the usual one. We refer to Remark \ref{rmk:equi_class} for more details on the specific nature of this moduli space.

For completeness, the cases with $2g-2+n\leq0$ are discussed in Section \ref{sec:2cone_sphere}.

\medskip
\textbf{Remark.} Studying the entire moduli space $\Msphgn$ is a much more challenging problem, which currently seems to be beyond our reach. In \cite[Section 6]{MonPa19}, it is asserted that when $\vec{\theta}$ satisfies the so-called “non-bubbling condition”, the real dimension of the moduli space of spherical surfaces with genus $g$ and $n$ prescribed cone angles is precisely $6g - 6 + 2n$. However, this non-bubbling condition seems to be irrelevant to the angle constraint for dihedral surfaces. On one hand, the angle vector for a co-axial surface never satisfies the non-bubbling condition. However, there do exist strict dihedral surfaces whose angle vectors satisfy
the non-bubbling condition.


\bigskip
Here is our strategy for proving the main theorem, using the strict dihedral case as an example (Theorem \ref{thm:dim_count_D}). A dihedral surface together with a pair of foliations is called a foliated surface (Definition \ref{def:LL_foli}).
Our first step is counting the dimension of the moduli space of foliated surfaces, with a prescribed angle vector and topological type of the foliations 
(Definition \ref{def:moduli_D} and Theorem \ref{thm:dim_count_DT}).
This involves studying the strip decomposition of a foliated surface (Definition \ref{def:strip_decomp}) and the concept of generic surfaces (Definition \ref{def:generic_foliated}). 
The dimension of this subset is determined using the connection matrix (Definition \ref{def:cnnct_mtr}), which captures the linear constraints on the geometric parameters of a strip decomposition.
Next, we apply the identification lemma (Subsection \ref{subsec:iden}) to pass this result to a subset of the original moduli space, by forgetting the foliation structure. The identification lemma shows that the choice of strip decomposition for a given strict dihedral surface is always finite (Proposition \ref{prop:id_SDhd}). 
Finally, combining our findings with the results in \cite{GT23}, we obtain the maximum dimension of moduli spaces among all possible topological types of foliations.

Figure \ref{fig:flow_chart} provides a graphical representation of the internal logical relations between the key concepts and main results. Additionally, we have a separate flow chart Figure \ref{Summary} at the end, summarizing the entire process. The co-axial case follows a similar proof strategy, with some technical modifications.

\begin{figure}[h]\centering
\begin{tikzpicture}[node distance=10pt, >=Stealth, auto]
    \tikzset{
      block/.style = {rectangle, draw, align=center, minimum height=1.6em, font=\footnotesize}, 
    }
    \node [block] (foli) {Definition \ref{def:LL_foli}};
    \node [block, below left =of foli, xshift=0cm] (strip) {Definition \ref{def:strip_decomp}};
    \node [block, below right=of foli, xshift=0cm] (annulu) {\ Definition \ref{def:ann_decomp}\ };
    \node [block, below=of strip, xshift=-1.5cm] (matrix) {Definition \ref{def:cnnct_mtr}};
    \node [block, below=of strip, xshift=1.5cm] (generic) {Definition \ref{def:generic_foliated}};
    \node [block, below=of annulu] (twist) {Definition \ref{def:twist}};
    \node [block, rounded corners, below=of generic, xshift=-1.5cm] (rank) {Proposition \ref{prop:rank}};
    \node [block, rounded corners, below=of generic, xshift=1.5cm] (lowdim) {Corollary \ref{cor:low_dim}};
    \node [block, rounded corners, below=of rank] (MsphT) {Theorem \ref{thm:dim_count_DT}};
    \node [block, rounded corners, below=of lowdim] (idlemma) {Proposition \ref{prop:id_SDhd}};
    \node [block, rounded corners, below=of MsphT] (Msph) {Theorem \ref{thm:dim_count_D}};
    \draw[->][rounded corners] (foli.east) -| (4,-2) |- (idlemma.east);
    \draw[->][rounded corners] (foli.west) -| (strip.north);
    \draw[->][rounded corners] (foli.east) -| (annulu.north);
    \draw[->][rounded corners] (strip.west) -| (matrix.north);
    \draw[->][rounded corners] (strip.east) -| (generic.north);
    \draw[->] (annulu) -- (twist);
    \draw[->][rounded corners] (matrix.south) |- (rank.west);
    \draw[->][rounded corners] (generic.south) |- (rank.east);
    \draw[->][rounded corners] (generic.south) |- (lowdim.west);
    \draw[->][rounded corners] (twist.south) |- (lowdim.east);
    \draw[->] (lowdim) -- (MsphT);
    \draw[->] (rank) -- (MsphT);
    \draw[->] (idlemma) -- (MsphT);
    \draw[->] (MsphT) -- (Msph);
\end{tikzpicture}
\caption{Leitfaden.}
\label{fig:flow_chart}
\end{figure}
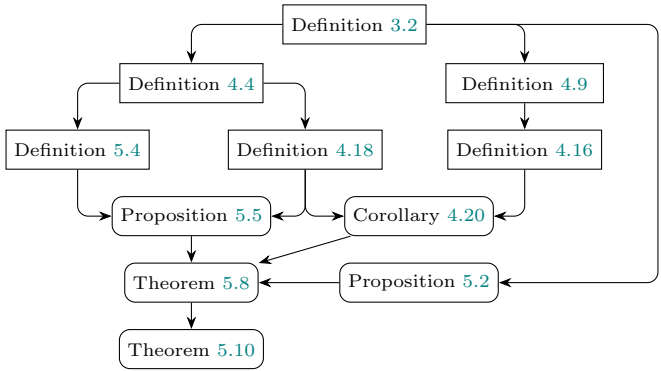
The paper is organized as follows.
\begin{itemize}
    \item Section 2 provides preliminaries on \csphsf s, abelian and quadratic differentials, and measured foliations. We also recall the definitions of co-axial and dihedral metrics.
    \item Section 3 introduces the main character, the longitude and latitude foliation on dihedral surfaces. Various moduli spaces are defined, preparing for the final dimension count. 
    \item Section 4 presents two mutually “perpendicular” geometric decompositions of dihedral surfaces based on the above two foliation structures. 
    A range of geometric deformations are presented, which can be considered analogues to the periodic coordinates used for translation surfaces.
    \item Section 5 is dedicated to dimension counting. We define the connection matrix and prove the identification lemma. The dimension counting for various moduli spaces is then carried out sequentially. 
\end{itemize}

\noindent \textbf{Acknowledgements.} The authors wish to extend their sincere appreciation to Guillaume Tahar for his outstanding series of lectures during his visit to USCT in August 2023 and a conversation in December 2023. The concepts of decomposition and sliding were significantly influenced by the enlightening discussions we had with him. We are also grateful to Jijian Song for his stimulating dialogues on this subject, and to Alexandre Eremenko, Dídac Martínez-Granado, Weixu Su, Dong Tan and Bing Wang for valuable suggestions. 
B.X. is supported in part by the Project of Stable Support for Youth Team in Basic Research Field, CAS (Grant No. YSBR-001) and NSFC (Grant Nos. 12271495, 11971450, and 12071449).

\section{Preliminaries} \label{sec:csphsf}
In our settings, the underlying Riemann surface is not predetermined. When we refer to a compact ``surface'', it encompasses both a topological surface, possibly with a metric, and a Riemann surface, depending on the specific context. And we prefer using ``\csphsf s'' rather than ``\csphmtr s''.

\subsection{Cone spherical surfaces}\label{sec:spherical_cone}
We are studying a Riemannian metric on a compact surface of constant curvature $K\in\{-1,0,+1\}$ with conical singularities.
This indicates that at each point, in geodesic polar coordinates, the metric is locally expressed as
\begin{equation*}
    \mathrm{d}s^2=\left\{   \begin{array}{lll}
        \mathrm{d}r^2 + \alpha^2 \sinh^2 r \mathrm{d}\theta^2 & , K=-1 & \quad \text{(hyperbolic geometry)} \\
        \mathrm{d}r^2 + \alpha^2 r^2 \mathrm{d}\theta^2 & , K=0 & \quad \text{(flat geometry)} \\
        \mathrm{d}r^2 + \alpha^2 \sin^2 r \mathrm{d}\theta^2 & , K=+1 & \quad \text{(spherical geometry)} \\
                \end{array}
    \right.
\end{equation*}
with the cone angle of that point being $2\pi\alpha>0$. Here $\alpha\neq1$ only at finitely many isolated points, which we refer to as \DEF{cone points}. In this paper we focus on the case of $K=+1$. 

\bigskip
There is also a geometric description. Recall some notations of \cite{MonPa16}.
Let $\mathbb{S}^2\subset\mathbb{R}^3$ be the unit sphere, parameterized by spherical coordinates
$$ \STwo = \big\{\ (\psi,\phi):=(\cos\psi\sin\phi, \sin\psi\sin\phi,\cos\phi) \ \big\rvert\ \psi\in[0,2\pi),\ \phi\in[0,\pi]\ \big\} $$
and $N=(0,0), S=(0,\pi)$ be the north and south pole.
Given $\alpha\in(0,1)$ and $r\in(0,\pi]$,
\[ \SYM{B_\alpha(r)} :=\big\{\ (\psi,\phi) \ \rvert\ \psi\in[0,2\pi\alpha],\ \phi\in[0,r)\ \big\} 
\]
is called a \DEF{sector} of angle $2\pi\alpha$ centered at $N$ with radius $r$. Let $P=(0, r), Q=(2\pi\alpha, r)$ be the other two vertices, and $NP, NQ$ be the two geodesic boundary radii of the sector.
For any $\theta>0$, let $M$ be the smallest positive integer such that $\theta/M<1$. We take $M$ copies of $B_{\theta/M}(r)$, denoted by $B^{(i)}$ with vertices $N^{(i)}, P^{(i)}, Q^{(i)}$, $i=1,\cdots, M$.
Now gluing the pair of radii $N^{(i-1)}Q^{(i-1)}$ and $N^{(i)}P^{(i)}$ by the $\phi$ parameter for each $i=\EMP{2},\cdots, M$, we obtain a topological sector denoted by \SYM{$B_{\theta}(r)$}, and obviously generalizing the previous one. This serves as a local model for cone vertices of polygons. See figure \ref{fig:Cone}.

If the remained two boundary radii $N^{(M)}Q^{(M)}, N^{(1)}P^{(1)}$ are further glued, we obtain a \DEF{cone} of angle $2\pi\theta$ with radius $r$, denoted by \SYM{$S_\theta(r)$}. This serves as a local model for cone points on surfaces.
In both $B_{\theta}(r)$ and $S_\theta(r)$, the unique image of all $N^{(i)}$'s is the cone point. Note that every point on $\STwo$ has a \nbhd\ isometric to $S_{1}(r)$ for some small $r$. Then we can consider a \csphsf\ as a compact metric surface such that every point has a neighborhood isometric to some $S_{\theta}(r)$, with $\theta=\theta_i$ at cone point $p_i$ and $\theta=1$ otherwise. This naturally induces a complex structure on the surface.

\begin{figure}[hbt]
  \centering
  \includegraphics[width=0.9\textwidth]{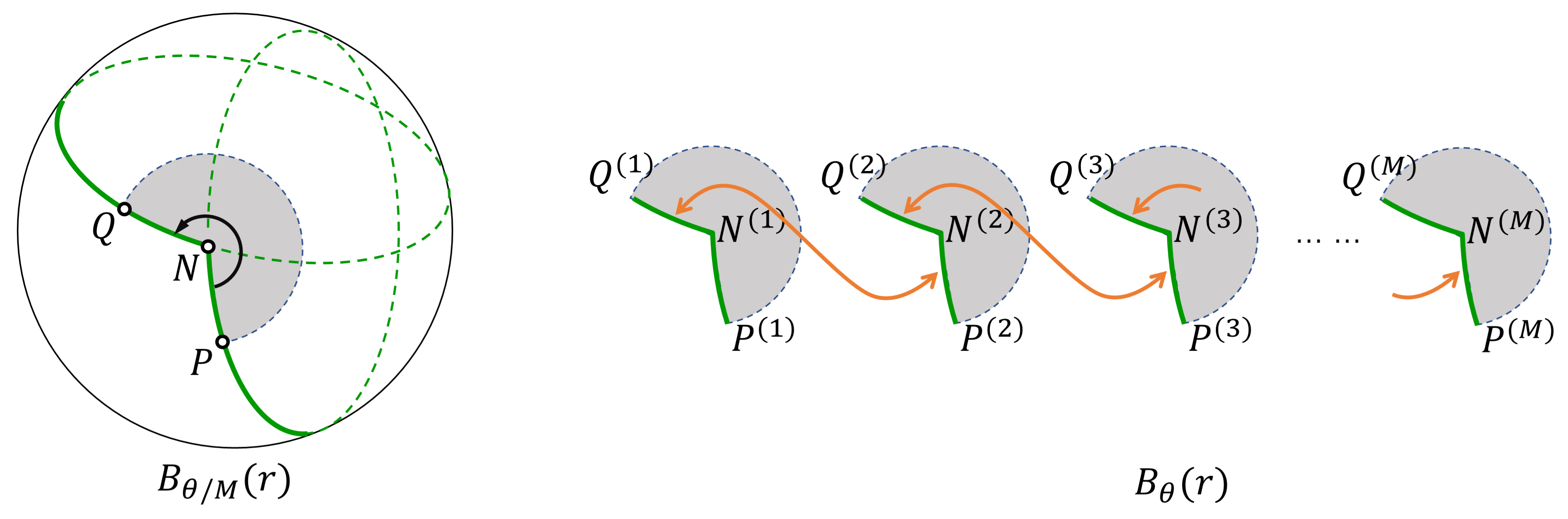}
  \caption{Gluing sectors of small angle to obtain a local model for cone vertex.}
  \label{fig:Cone}
\end{figure}

When $r=\pi$, the closure $\SYM{B_\theta}:=\overline{B_\theta(\pi)}$ is called a spherical \DEF{bigon} of angle $2\pi\theta$, or a bigon of width $\theta$ (see Section \ref{sec:strip_decomp}). It is bounded by two geodesic segments of length $\pi$, and homeomorphic to a closed disk. It is foliated by a continuous family of disjoint geodesic arcs of length $\pi$ connecting the two cone points. 
$\SYM{S_\theta}:=\overline{S_\theta(\pi)}$ is called a \DEF{football} of angle $2\pi\theta$. It is a topological sphere with two cone points of the same angle, so it is also called a foliated singular sphere in \cite{Tah22}.

\subsection{Developing map, monodromy and dihedral surface}\label{sec:dihedral}
We will frequently switch between the unit sphere $\STwo$ and the Riemann sphere $\CC\cup\{\infty\}$. 
The isometry group $\SOthr$ of the unit sphere is identified with
$$ \PSUtwo=\left\{\ z\mapsto \frac{az+b}{-\bar{b}z+\bar{a}} \ \bigg\rvert\ a,b\in\CC, \abs{a}^2+\abs{b}^2=1 \ \right\} \big{/} \pm I \subset \mathrm{PGL}(2,\CC) $$
which acts on the Riemann sphere. The standard spherical metric on the Riemann sphere is given by $g_{\rm FS} = 4\lvert \mathrm{d}w \rvert ^2/(1+\lvert w \rvert ^2)^2$.

Given a \csphsf\ $\Sigma$, its metric naturally induces a complex structure on it. There exists a multi-valued meromorphic function $D$ on $\Sigma$, branched at the cone points in $P$. Furthermore, the pullback metric $D^*(g_{\rm FS})$ of $g_{\rm FS}$ is coincident with the \csphmtr\ on $\Sigma$. Such a map $D$ is called a \DEF{developing map}.
It has monodromy in $\PSUtwo$. This means that with a base point $z_0\in \Sigma- P$, for any closed curve $\gamma\subset\Sigma- P$ passing through $z_0$, the analytic continuation $D(\gamma\cdot z)$ of $D$ along $\gamma$ can be expressed as $D(\gamma\cdot z) = \rho_D(\gamma)\big( D(z) \big)$, where $\rho_D(\gamma)\in\PSUtwo$. The element $\rho_D(\gamma)$ depends only on the homotopy class of $\gamma$. Consequently, a developing map $D$ induces a \DEF{monodromy representation} $\rho_D:\pi_1(\Sigma-P,z_0) \to \PSUtwo$ and its image is called \DEF{monodromy group}.
The developing map is an orientation preserving local isometry, unique up to post-composing by an element of $\PSUtwo$. Similarly, the monodromy representation and group are unique up to conjugation by an element in $\PSUtwo$.
See \cite{UYa00, Erm04, CWWX15, SCLX20} for more about the analytic aspect of developing map.

The developing map can also be defined geometrically \cite{MonPa16, EMP20, MP23}, behind which the general idea is the concept of $(G, X)$-manifold \cite[Chapter 3]{Thu3D}. 
Let $\pi:\tilde{\Sigma}' \to \Sigma -P$ be the universal cover of the punctured surface $\Sigma -P$. Then the developing map $D$ is a well-defined map from $\tilde{\Sigma}'$ to $\STwo$. 
This map extends to the metric completion $\hat{\Sigma}$ of $\tilde{\Sigma}'$, which we continue to call a developing map and denote it by $D$. 
$\hat{\Sigma}$ also admits a $\pi_1(\Sigma-P, z_0)$-action and $D$ remains to be $\rho_D$-equivariant, namely 
$D(\gamma\cdot z) = \rho_D(\gamma)\big( D(z) \big)$ for all $z\in\hat{\Sigma}$ and $\gamma\in \pi_1(\Sigma-P, z_0)$.
The covering map also extends and we continue to denote it by $\pi:\hat{\Sigma} \to \Sigma$. See the diagram below.
Points in $\hat{\Sigma} - \tilde{\Sigma}'$ are the preimages $\pi^{-1}(P)$, and each of them corresponds to a homotopy class of primitive loops passing through the base point $z_0$ that surround some cone point in $P$. The monodromy representation of such a homotopy class, as an element in $\SOthr$, fixes the image of the added point. Around a point in $\hat{\Sigma} - \tilde{\Sigma}'$, the extended developing map $D$ acts like an infinite covering of the unit disk branched at the origin. 
A detailed version in hyperbolic geometry is provided in \cite[Section 2]{ELS22}, 
and can be easily transplanted to spherical geometry. Using the concept of $(G,X)$-manifold, we have $X=\STwo$ and $G=\SOthr \cong \PSUtwo$ in our setting. 
\[\xymatrix@C=1mm{
    \pi_1(\Sigma-P) \curvearrowright & \widetilde{\Sigma}' \ar[d]_{\pi} & \subset & \ \widehat{\Sigma} \ \ar[d]_{\pi} \ar[rr]^{D}
    & \quad\quad & \mathbb{S}^2 & \curvearrowleft  \rho(\pi_1)   
    \\
     & \Sigma-P & \subset & \ \Sigma\  & & \\
    }
\]

Typically, the monodromy group of a \csphsf\ can be quite complicated. However, there are special classes of surfaces with simpler monodromy. These surfaces are highly intriguing as they demonstrate a significant connection with meromorphic differentials and foliations on Riemann surfaces, making them our primary focus of study. We shall denote the monodromy group by $\rho(\pi_1)$ in short.

\begin{definition}\label{def:coaxial}
    A \csphsf\ is called \DEF{co-axial} if its monodromy group fixes an axis of ${\Bbb S}^2$. In other words, $\rho(\pi_1)$ is conjugated to a subgroup of
	$$ \Uone := \left\{\ z\mapsto \mathrm{e}^{\ii\theta}z = \frac{\mathrm{e}^{\ii\theta/2}z}{\mathrm{e}^{-\ii\theta/2}} \ \bigg\rvert\ \theta\in\RR \ \right\} \cong \mathrm{SO}(2). $$	
	This concept is equivalent to \DEF{reducible metric} in \cite{UYa00}.
\end{definition}

This monodromy condition is slightly generalized in \cite{SCLX20, GT23}.
\begin{definition}\label{def:dihedral}
    A \csphsf\ is called \DEF{dihedral} if its monodromy group preserves a pair of antipodal points on the unit sphere ${\Bbb S}^2$. Equivalently, $\rho(\pi_1)$ is conjugated to a subgroup of
    $$ \Otwo := \left\{\ z\mapsto \mathrm{e}^{\ii\theta}z,\ z\mapsto \frac{\mathrm{e}^{\ii\theta}}{z} \ \bigg\rvert\ \theta\in\RR \ \right\} \cong \ZZ_2 \ltimes \Uone. $$
    A dihedral surface that is not co-axial is called \DEF{strict dihedral}.
\end{definition}

\newcommand{\Dev}{\overline{D}}
\newcommand{\CUni}{\widetilde{\Sigma_\mathfrak{c}}}

\subsection{Differentials and dihedral surfaces}
Let $X$ be a compact Riemann surface of genus $g\geq0$ and $K_X$ be its canonical line bundle. A meromorphic \DEF{abelian differential}, denoted as $\omega$, is a meromorphic 1-form on $X$. A meromorphic \DEF{\qdf}, denoted as $q$, is a meromorphic section of $K_X^{\otimes 2}$. 
In a coordinate chart, the differential can be expressed as $\omega(z)\mathrm{d}z$ or $q(z)\mathrm{d}z^2$, where $\omega(z)$ and $q(z)$ are meromorphic functions compatible with coordinate transformations.
The zeros and poles of a differential are called the \DEF{critical points}, while all other points are \DEF{regular points}. 
We will only consider abelian differentials with at worst simple poles, and \qdf s with at worst double poles.
A \qdf\ is locally the square of an abelian differential near regular points, 
but this property does not extend globally. We call a \qdf\ \DEF{primitive} if it is not the square of an abelian differential globally.

A differential on a Riemann surface defines a conformal cone flat metric. For a regular point, it is always possible to choose coordinates where the differential is expressed simply as $\mathrm{d}z$ or $\mathrm{d}z^2$. Near a simple pole of an abelian differential or a double pole of a \qdf, the metric takes the forms of a half-infinite cylinder. All the other critical points become cone points of angle $k\pi (k\in\NN_+)$ in this metric. We refer to \cite{Strebel, HubbardV1} for more about \qdf s.

\bigskip
Now let $\Sigma$ be a dihedral surface with a developing map $D$. 
By composing a $\PSUtwo$ element, we may assume that $\rho(\pi_1)\subset \Otwo$. Note that the \qdf\ $\mathrm{d}z^2/z^2$ on the Riemann sphere is $\Otwo$-invariant, and that the developing map $D$ is $\rho$-equivariant, $\pi_* D^* \left( \mathrm{d}z^2/z^2 \right)$ is a well-defined differential on $\Sigma$. 
For co-axial surfaces, we can pull back the 1-form $\mathrm{d}z/z$. 
Conversely, a meromorphic \qdf\ $q$ induces a dihedral \csphmtr\ if and only if the integral of $\sqrt{q}$ is real along any closed curve on $\Sigma-P$ and for any choice of analytic branch of the square root \cite{Li19,SCLX20}. 
The “real” property of $q$ can be verified as well-defined. 

This observation establishes a connection between \csphsf s, \qdf s and cone flat surfaces, which is of significant importance to the current work. The subsequent parts will provide a geometric-topological perspective. 
For more details on dihedral surfaces and differentials, or the complex analytical aspect of this relationship, see \cite{CWWX15, SCLX20, GT23}. 

\subsection{Measured foliations}
The topological properties of an abelian or quadratic differential can be understood through a pair of transverse \mfoli s induced by it. In this paper, we focus on compact surfaces without boundary. 
We refer to \cite[Expos\'{e} 5]{FLP} or \cite[Section 2]{PapTh07} for further details on topological \mfoli s. A comprehensive understanding of connections between differentials, foliations, and flat geometry can be found in \cite[Chapter 5]{HubbardV1}. 
A \DEF{foliation} on a surface is a decomposition of the 2-dimensional surface into 1-dimensional immersed manifolds. 
There may be isolated points, called \DEF{singularities}, where this decomposition is not well-defined. 
Alternatively, in the smooth category, a foliation can be understood as a line field with a set of isolated singular points \cite[Part Two, Chatper \uppercase\expandafter{\romannumeral3}]{Hopf83}.
The 1-dimensional manifolds, or the trajectories of the line field, are called \DEF{leaves}. The singularities are the points where multiple leaves intersect. Leaves that meet singularities are called \DEF{singular leaves}, while those without any singularity are called \DEF{regular leaves}. A foliation is \DEF{orientable} if it possesses a globally continuous orientation for all its leaves outside singularities.

According to Definition 1.2 of \cite[Part Two, Chatper \uppercase\expandafter{\romannumeral3}]{Hopf83}, each singularity of a foliation can be assigned with an index, which measures the total change in angle of the tangent line when going around a small primitive loop around the singularity. This index is akin to the one used for zeros of a vector field, typically taking on a half-integer value. 
For convenience, the \DEF{index} $\SYM{\mathrm{Ind}(p)}$ of a singularity $p$ of a foliation here is defined as twice the index specified in \cite{Hopf83}. 
In particular, when $\mathrm{Ind}(p)\leq1$, $p$ is an \DEF{$s$-prong}, where $s=2-\mathrm{Ind}(p)$ denotes the number of leaf segments converging at that singularity. Notably, a singularity of index $0$ is actually regular. This definition aligns with the one in \cite[Expos\'{e} 5]{FLP} and is equivalent to ``projective index'' in \cite{CrGr17}. 
Figure \ref{fig:index} provides a local picture of singularities of index $+2,+1,0,-1$. 
We will focus on singularities with index $\leq2$.

\begin{figure}[ht]
  \centering
  \includegraphics[width=\textwidth]{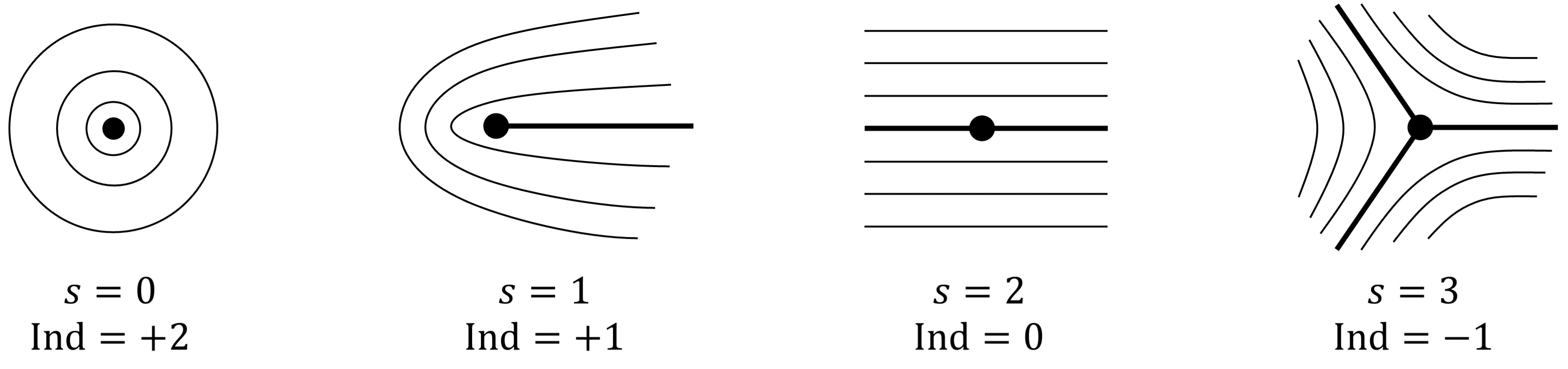}
  \caption{Examples of singularities with various indexes.}
  \label{fig:index}
\end{figure}

The indices of singularities provide global information about the topology of surface. We can derive an index formula that is analogous to the \PH\ theorem for vector fields. This formula can also be viewed as a topological variant of the Gauss-Bonnet theorem.
See \cite[Proposition 5.1]{FLP} or \cite[Part Two, Chatper \uppercase\expandafter{\romannumeral3}, 2.2 Theorem \uppercase\expandafter{\romannumeral2}]{Hopf83} for a detailed proof. Also see \cite{CrGr17} for a rigorous statement and complete proof that holds for line fields in all dimension.

\begin{theorem}[\PH\ Theorem]\label{thm:PH}
    Suppose $F$ is a foliation on a closed surface $S$ of genus $g\geq0$. Then
    \begin{equation}\label{eq:PH}
    \sum_{\mathrm{singularities}\ p} \mathrm{Ind}(p) = 2 (2- 2g).
    \end{equation}
\end{theorem}

A \DEF{measured foliation} on a closed surface $S$ is a foliation with a transverse measure that assigns each transverse arc a positive measure. 
The measures assigned to the transverse arcs are invariant under isotopies, meaning that during the isotopy, where each point remains on the same leaf, the measures of the transverse arcs remain unchanged.

Usually, we consider the equivalence classes of measured foliations based on two types of transformations. The first type is induced by homeomorphisms of the surface that are isotopic to the identity. The second type is called Whitehead move. This kind of deformation takes place in a neighborhood of a singular leaf starting and ending at singularities, and collapses the leaf to a point. 
However, in our settings, every singularity represents a cone point. So Whitehead moves are \emph{not} allowed within finite time since they would result in the merging of cone points.

\bigskip
Each \qdf\ $q$ on a Riemann surface $X$ gives rise to a pair of transverse \mfoli s $(H_q, V_q)$. This means that $H_q, V_q$ share the same set of singularities and their respective indexes match. Moreover, outside the singularities, the leaves of $H_q$ and $V_q$ intersect transversely. 
Let us focus on the case where $q$ contains at worst poles of order 2. 
The basic case is $X=\CC$ and $q=\mathrm{d}z^2$. A horizontal trajectory of $q$ corresponds to a horizontal line with constant $y$-coordinate. All these horizontal trajectories form the horizontal foliation $H_q$. The transverse measure on a transverse arc of $H_q$ is the total variation of $y$-coordinate, given by $\abs{\mathrm{d}y}=\abs{\mathrm{Im}\sqrt{q}}$. The vertical foliation $V_q$ is defined similarly, with the $y$-coordinate replaced by the $x$-coordinate. 
For a \qdf\ on a general Riemann surface, at any regular point, one can always choose some suitable coordinate chart such that $q$ is locally expressed as $\mathrm{d}z^2$. 
At a critical point $z_0$ of $q$, there exists a coordinate chart such that $z_0$ is mapped to $0$ and locally $q=c z^k\mathrm{d}z^2$, where $c\in\CC^*$ is constant and $k\geq-2$. Then $z_0$ is a singularity of index $(2-k)$ for both foliations.
Note that around a double pole, all horizontal trajectories converge to the pole, while all vertical trajectories form closed cycles around the pole. 
For an abelian differential $\omega$, the induced foliations $(H_\omega, V_\omega)$ are all orientable, since the trajectories are integral curves of vector fields.
We refer to \cite{Strebel}, especially Chapter \uppercase\expandafter{\romannumeral3}, for more details.

\textbf{Remark.} Strebel differential, whose regular leaves of $V_q$ are all closed curves around exactly one pole, is a special type of differentials that induces dihedral metrics. They correspond to hemispherical surfaces in \cite[Section 2.5]{GT23}, and are related to limit surfaces of grafting along a multicurve \cite[Section 5.2]{Gup14}.

\section{The latitude and longitude foliation}\label{sec:LLfoli}
Now we consider a class of \mfoli s slightly adjusted from the one directly induced by differentials. These foliations provide more geometric information, making them highly useful in the study of dihedral \csphsf s.

\subsection{On the unit sphere}\label{sec:foli_sphere}
\newcommand{\latifoli}{latitude foliation}
\newcommand{\longfoli}{longitude foliation}
\newcommand{\FoliLati}{\mathcal{F}}
\newcommand{\FoliLogi}{\mathcal{G}}
\begin{definition}\label{def:LL_foli_sphere}
    There is a natural pair of transverse measured foliations on $\STwo$.
    \begin{itemize}
        \item The \DEF{latitude foliation} \SYM{$\mathcal{F}$} consists of all latitude circles of $\STwo$. The measure of a transverse arc is given by the total variation of latitude.
        \item The \DEF{longitude foliation} \SYM{$\mathcal{G}$} consists of all meridian lines of $\STwo$. The measure of a transverse arc is given by the total variation of longitude.
    \end{itemize}
\end{definition}
All leaves are lines with constant $\psi$ or $\phi$ coordinate. Leaves of $\FoliLati$ are the level sets of Morse function $\cos\phi$ on $\STwo$. Notably, the equator is the only geodesic leaf within $\FoliLati$. In comparison, leaves in $\FoliLogi$ are geodesics of length $\pi$, connecting the two poles. The poles $N,S$ are the only singularities in both foliations, with index $+2$.

\bigskip
Consider the \DEF{conformal Mercator map} $Mer:\CC \to \STwo-\{N,S\}$ given by
    \begin{equation}
        Z=\psi+\ii Y \longmapsto \left( \psi, 2\arctan(\mathrm{e}^Y) \right).
    \end{equation}
This map serves as the conformal universal cover of the twice punctured sphere. Under this map, $\FoliLogi$ is exactly the pullback of the vertical measured foliation $\mathrm{d}\psi$, while $\FoliLati$ is the pullback of horizontal foliation with the measure given by $\mathrm{d}Y/\cosh Y$. 
By using the coordinate $z=r\mathrm{e}^{\ii\theta}$ on the Riemann sphere, $Mer:\CC\to\CC^*$ is given by
    \begin{equation}
        Z = \psi+\ii Y \longmapsto \exp(\ii Z) = \mathrm{e}^{-Y}\mathrm{e}^{\ii\psi}.
    \end{equation}
Then $\FoliLogi, \FoliLati$ is given by $\mathrm{d}\theta$ and $2\mathrm{d}r/(1+r^2)$ respectively. 
Hence, the pair of foliations $(\FoliLati, \FoliLogi)$ is topologically equivalent to the ones induced by the differential $\ii \mathrm{d}z / z$ on the Riemann sphere (or simply $\mathrm{d}Z$ on the covering $\CC$) with a modified measure. The purpose of this modification is to record spherical geometric quantities.

\subsection{On dihedral surfaces}\label{sec:ll_foli}
It can be observed that the differentials $\mathrm{d}\theta$ and $2\mathrm{d}r/(1+r^2)$ are fixed by $\Uone$-action, and remain invariant up to signature under $\Otwo$-action. Therefore, as \mfoli s, the latitude foliation $\FoliLati$ and the longitude foliation $\FoliLogi$ are both fixed by the $\Otwo$-action.
Recall that in Section \ref{sec:dihedral}, $\hat{\Sigma}$ is the metric completion of the universal covering space of the punctured surface $\Sigma-P$. A developing map $D:\hat{\Sigma}\to\STwo$ is a $\rho$-equivariant map, and $\pi:\hat{\Sigma}\to\Sigma$ extends the universal covering map.

\begin{definition}\label{def:LL_foli}
Let $\Sigma$ be a dihedral \csphsf. 
\begin{enumerate}
    \item Let $D$ be a developing map such that $\rho(\pi_1)\subset \Otwo$. 
    Then $\SYM{F}:=\pi_* D^*(\FoliLati)$ and $\SYM{G}:=\pi_* D^*(\FoliLogi)$ are well-defined \mfoli s on $\Sigma$.
    $F$ and $G$ are also called \DEF{latitude and longitude foliation} of $\Sigma$ with respect to $D$, regarded as horizontal and vertical foliation.
    \item 
    $\Sigma$ together with a pair of latitude and longitude foliation $(F,G)$ is called a \DEF{foliated (dihedral) surface}, denoted by a triple \SYM{$(\Sigma;F,G)$}.
\end{enumerate}
\end{definition}

It is important to note that there may be different choices of $(F,G)$ on a given $\Sigma$. See Example \ref{eg:three_cone_sphere}, \ref{eg:one_cone_torus}.
The latitude foliation is also defined in \cite[Section 2.2]{GT23}

The function $\abs{\cos\phi}$ is $\Otwo$-invariant on $\STwo$. So for any dihedral surface $\Sigma$ with developing map $D$ as above, $\abs{\cos\phi}$ can induce a function $H:\Sigma\to[0,1]$, called the \DEF{absolute latitude function}. The leaves of latitude foliation $F$ are the level sets of $H$. For co-axial surfaces, we can pull back the $\Uone$-invariant function $\cos\phi$ to obtain a Morse function $h:\Sigma\to[-1,1]$. This is similar to the function $\Phi$ introduced in \cite{CWWX15,WWX22axv} and we have $2h+2=\Phi$.

\subsection{Local information of singularities}\label{sec:singular_info}
Let $(\Sigma;F,G)$ be a foliated dihedral surface induced by a developing map $D$.
Let $N,S$ be the north and south poles, and $\{\phi=\frac\pi2\}$ be the equator of $\STwo$. Both sets $\{N,S\}$ and $\{\phi=\frac\pi2\}$ are $\Otwo$-invariant.
\begin{definition}\label{def:singular_info} ~
\begin{enumerate}
    \item The set $\pi (D^{-1} \{\phi=\frac\pi2\}) \subset \Sigma$  
    is called the \DEF{equatorial net} of the foliated surface or both foliations $F,G$, denoted by \SYM{$\mathcal{E}$}.
    \item Points in $\SYM{Q}:=\pi (D^{-1}\{N,S\}) $ 
    are called the \DEF{poles} of both foliations. 
    This is the set of singularities with index $+2$.
    \item Other singularities are called the \DEF{zeroes} of both foliations, with all indexes $\leq +1$.
\end{enumerate}
\end{definition}

Equivalently, using the absolute latitude function $H$, we have $\mathcal{E}=H^{-1}(0)$, $Q=H^{-1}(1)$. Note that the equatorial net $\mathcal{E}$ is the only geodesic leaves of $F$.


\bigskip
By the geometry of $\Otwo$ actions on $\STwo$, the cone angle depends on the position of a cone point. 
As discussed in Section \ref{sec:dihedral}, the monodromy representation of a simple loop surrounding a cone point $p\in P$ is a rotation. One of the fixed points of this rotation is the developing map image of the point in $\hat{\Sigma} - \tilde{\Sigma}'$ corresponding to the loop. 
When $p$ lies on the equatorial net, the developing map image of the point lies on the equator of $\STwo$.
Since $\Otwo$ action preserves the equator, the representation of the simple loop must either be trivial or a $\pi$-rotation. And consequently, the cone angle of $p$ is an integer multiple of $\pi$. 
When $p$ is a pole of the foliations, the representation of the simple loop can be an arbitrary rotation that fixes the two poles. Finally, if $p$ is not a pole or a point on the equatorial net, the representation of the simple loop must be trivial and the cone angle is an integer multiple of $2\pi$. 
For a co-axial surface, if $p$ is not a pole, its cone angle must be an integer multiple of $2\pi$, since $\pi$-rotation about a point on the equator is not allowed. 
So every cone point must be a singularity of the foliations, but a pole singularity may be a smooth point of the metric. 
Also see \cite[Theorem 1.2]{SCLX20} for a special class of dihedral surfaces and \cite[Theorem 1.5]{CWWX15} for co-axial surfaces.

We summarize the local information of singularities of the foliations in Table \ref{tab:singularity}. Also see \cite[Section 2.2]{GT23}. The poles provide geometric information, while the zeros provide more topological information. We also introduce the figure legend in the final row for illustrating a foliated surface. The zeros are depicted as solid circles, while poles that are not cone points are depicted as hollow circles. If a pole is also a cone point, it is shown with a hollow circle overlapping a solid circle. 
\begin{table}[h] 
\centering
\normalsize
\renewcommand\arraystretch{1.4}
\begin{tabular}{c||c|c|c|c}
    singularity $x$ & $\in P \cap Q$ & $\in P \cap \mathcal{E}$ & $\in P - (\mathcal{E}\cup Q)$ & $\in Q-P$ \\
    \hline
    cone angle & (no restriciton) & $k \pi$ & $ 2k \pi $ & (smooth point) \\
    \hline
    index & $+2$ & $2-k$ & $2-2k$ & $+2$ \\
    \hline
    figure legend 
    & \rlap{$\raisebox{-0.9ex}{\text{\larger[5]\ensuremath{\bullet}}}$} $\raisebox{-0.9ex}{\text{\larger[5]\ensuremath{\circ}}}$ 
    & \multicolumn{2}{c|}{ $\raisebox{-0.9ex}{\text{\larger[5]\ensuremath{\bullet}}}$ } 
    & $\raisebox{-0.9ex}{\text{\larger[5]\ensuremath{\circ}}}$
\end{tabular}
\vspace{4pt}
\caption{Local information of singularities for dihedral surfaces. $P$ is the set of all cone points, $Q$ is the set of all poles, and $\mathcal{E}$ is the equatorial net.
The angle of a cone point in $P\cap Q$ can be expressed by the measure. }
\label{tab:singularity}
\end{table}

\subsection{More about the measure}\label{sec:bt_meas}
The total measure of leaf segments in $F,G$ is most concerned. We give a brief description next, which aligns with the general theory of \mfoli.

An embedded regular leaf segment $\gamma:[0,1]\to\Sigma$ of $G$ 
is a geodesic arc. The \DEF{total measure} of $\gamma$ under $F$, also referred to as the \DEF{intersection number} of $F$ and $\gamma$, is defined as the spherical length of $\gamma$, denoted by $i(F,\gamma)$.

Similarly, consider another embedded regular leaf segment $\delta:[0,1]\to\Sigma$ of $F$. 
Let $\tilde{\delta}$ be a lifting of $\delta$ through the covering $\pi:\hat{\Sigma}\to\Sigma$. 
The total measure of $\delta$ under $G$, or the intersection number $i(G,\delta)$, is defined as the total longitude variation of $D({\tilde{\delta}})$ on $\STwo$. 
Alternatively, we can consider the \DEF{equatorial projection}
    $$ \mathrm{pr}:\STwo-\{N,S\} \longrightarrow \{\phi=\frac\pi2\}\subset\STwo,\quad (\psi,\phi) \mapsto (\psi,\pi/2)=(\cos\psi, \sin\psi, 0). $$
Then $i(G,\delta)$ is equal to the spherical length of the curve $\mathrm{pr}\circ D(\tilde{\delta})$. In particular, when $\delta$ is a closed leaf of $F$ surrounding a pole $q$, $i(G,\delta)$ is equal to the cone angle $2\pi\theta_i$ if $q=p_i\in P$, and equal to $2\pi$ if $q\notin P$. This is analogous to the residue of a quadratic differential at a double pole, multiplied by $2\pi$.

\subsection{Various moduli spaces}\label{sec:moduli_spaces}
Now we introduce several kinds of moduli spaces used for the dimension count. As before, let $\vec{\theta}:=(\theta_1, \cdots, \theta_n) \in \RR_+^N$ be an angle vector with $\theta_i\neq 1$. 
\begin{definition}\label{def:labeled} ~ \begin{enumerate}
    \item A \DEF{labeled surface} is a \csphsf\ whose cone points are labeled as $p_1,\cdots,p_n$ such that $p_i$ has cone angle of $2\pi \theta_i$, for $i=1,\cdots,n$.
    \item Let $\Sigma_k\ (k=1,2)$ be two labeled \csphsf s, with cone points labeled as $p^k_i, i=1,\cdots,n$. They are \DEF{equivalent with label} if there exists an orientation-preserving isometry $I:\Sigma_1 \to \Sigma_2$ such that $I(p^1_i)=p^2_i$ for all $i$.
    \item The \DEF{labeled moduli space} \SYM{$\Msphgn$} is all equivalence classes of genus $g$ labeled surfaces with the given angle vector $\vec{\theta}$.
    \end{enumerate}
\end{definition}
From now on, all \csphsf s with prescribed cone angles will be treated as labeled surfaces.


\medskip
By Section \ref{sec:singular_info}, each cone point of a dihedral surface must be a singularity of any foliation. However, it can be either a zero or a pole. Therefore, we can consider foliated surfaces with prescribed topological type for each cone point. Some of the following notations are inherited from \cite{GT23}.

\begin{definition}
    Let $\vec{\theta}=(\theta_1, \cdots, \theta_n) \in \RR_+^N$ be an angle vector with $\theta_i\neq 1$. A \DEF{(topological) type partition} is a partition $\SYM{\vec{T}}=( E, O, N )$ of $\{1,\cdots,n\}$ satisfying that
\begin{itemize}
    \item $\forall i\in E, \theta_{i}\in \ZZ_{>1}$;
    \item $\forall j\in O, \theta_{j}\in \frac12 + \ZZ$.
\end{itemize}
    Let $ \#E=n_E, \#O=n_O, \#N=n_N $, then $n_E + n_O + n_N = n$.

    A foliated surface $(\Sigma;F,G)$ is \DEF{of type partition $\vec{T}$}, if $p_i$ is a zero with index $(2-2\theta_i)$ for all $i \in E \cup O$, and $p_k$ is a pole for all $k \in N$.
\end{definition}

To achieve a more precise statement of dimension count, we will consider co-axial and strict dihedral surfaces separately.

\begin{definition}\label{def:moduli_D}
    Let $\vec{T}$ be a type partition as above.
    \begin{enumerate}
    \item The moduli space \SYM{$\MsphgnD$} is the set of all equivalent classes of \EMP{strict dihedral} \csphsf s in $\Msphgn$.
    \item \SYM{$\MsphgnDF$} is the set of all foliated strict dihedral surfaces $(\Sigma;F,G)$ with $\Sigma\in\MsphgnD$.
    \item $\SYM{\mathcal{M}sphF_{g,n}^D(\vec{\theta};\vec{T})} \subset \MsphgnDF$ is the set of foliated strict dihedral surfaces of type partition $\vec{T}$.
    \item $\SYM{\mathcal{M}sph_{g,n}^D(\vec{\theta};\vec{T})} :=\{ \Sigma \ \lvert\ (\Sigma;F,G)\in \MsphgnDFT \}$ is the set of all equivalent classes of (unfoliated) strict dihedral surfaces of type partition $\vec{T}$.
\end{enumerate}
\end{definition}

The relations between moduli spaces defined above are summarized in the following diagram. In the diagram, all horizontal arrows are embeddings, while all vertical arrows are onto projections. We will show that the fibers of projections are always finite, and different type partition may lead to different dimension. 
\begin{equation}\label{eq:moduli_spaces_D} \begin{gathered}
    \xymatrix{
    \MsphgnDFT \ar@{^(->}[r]^{\quad\iota_F} \ar@{->>}[d]^{p_T}    & \MsphgnDF \ar@{->>}[d]^{p}    &\\
    \MsphgnDT \ar@{^(->}[r]^{\quad\iota_T}     & \MsphgnD \ar@{^(->}[r]^{\iota}     & \Msphgn
    }
\end{gathered} \end{equation}

The treatment for co-axial surfaces is slightly different. Note that a cone of angle $(2k+1)\pi$ can not be a zero of latitude foliation in this case. So we always assume $O=\emptyset$ in the type partition for co-axial surfaces.

\begin{definition}\label{def:moduli_C} ~
    \begin{enumerate}
        \item \SYM{$\MsphgnC$} is the set of all equivalent classes of co-axial \csphsf s in $\Msphgn$.
        \item \SYM{$\MsphgnCF$} is the set of all foliated co-axial surfaces $(\Sigma;F,G)$ with $\Sigma\in\MsphgnC$ and \EMP{$(F,G)$ orientable \mfoli s}. 
        \item $\SYM{\MsphgnCFT}\subset \MsphgnCF$ is the set of all foliated co-axial surfaces of type partition $\vec{T}$.
        \item $\SYM{\MsphgnCT} :=\{ \Sigma \ \lvert\ (\Sigma;F,G)\in \MsphgnCFT \}$ is the set of all (unfoliated) co-axial surfaces of type partition $\vec{T}$.
    \end{enumerate}
\end{definition}
The main distinction is that we will only consider orientable foliations on co-axial surfaces. Example \ref{eg:three_cone_sphere} shows that a co-axial surface can also be foliated by non-orientable $(F,G)$. This demonstrates the reasonableness of this definition and simplifies problems.
As before, we will study projections in the following diagram. Although the fiber of projections may have positive dimension, this does not affect our dimension count. 

\begin{equation}\label{eq:moduli_spaces_C} \begin{gathered}
    \xymatrix{
    \MsphgnCFT \ar@{^(->}[r]^{\quad\iota_F} \ar@{->>}[d]^{p_T}    & \MsphgnCF \ar@{->>}[d]^{p}    &\\
    \MsphgnCT \ar@{^(->}[r]^{\quad\iota_T}     & \MsphgnC \ar@{^(->}[r]^{\iota}     & \Msphgn
    }
\end{gathered} \end{equation}

\section{Decomposition and deformations of dihedral surfaces}
In this section, we present two types of canonical geometric decompositions of foliated surfaces, which are derived from the study of flat surfaces.
Similar ideas can also be found in previous works \cite{Erm20} and \cite{Li19}. Here we focus more on the \mfoli\ and spherical geometric structure rather than complex analysis aspects.

\subsection{Strip decomposition}\label{sec:strip_decomp}
This is an analogous of vertical strip decomposition for \qdf\ on the unit disk \cite[Section 19]{Strebel}. Also compare this subsection with \cite[Section 4]{WWX22axv}.

We start with the topology of longitude foliation of a foliated surface $(\Sigma;F,G)$, where $\Sigma$ is a dihedral surface of genus $g$ with $n$ cone points.

\newcommand{\SingF}{\mathrm{Crit}(F)}
\newcommand{\SingG}{\mathrm{Crit}(G)}
\begin{definition}\label{def:Sing_G}
    The set of \DEF{critical leaves} of $G$, denoted by \SYM{$\SingG$}, is defined as the union of all leaves of $G$ starting from $P-Q$, which are the zeros of the foliation.
\end{definition}

\begin{lemma}\label{lem:Sing_G}
    For a foliated surface $(\Sigma;F,G)$, all leaves in $G-\SingG$ are geodesic segments of length $\pi$, ending at points in $Q$.
\end{lemma}
\begin{proof}
    All leaves of $\FoliLogi$ on $\STwo$ are geodesics of length $\pi$, and the developing map $D$ is locally isometric. So every leaf of $G$ on $\Sigma$ must meet a pole within length $\pi$, or stop at a zero before meeting a pole. If a leaf never hits a zero, the developing map along any lifting of this leaf is an isometric immersion. So it is isometric to a leaf of $\FoliLogi$. This also shows that there is no recurrent leaf in $G$.
\end{proof}

There are only finitely many cone points, so the majority of leaves of $G$ terminate at poles in both directions. The leaves that end at zeros of 
$G$ indicate the genuine singularity of the foliation.

\begin{lemma}
    Whenever $2g-2+n>0$, $\SingG$ is non-empty.
\end{lemma}
\begin{proof}
    Based on to Table \ref{tab:singularity} and \PH\ formula (\ref{eq:PH}), for $g>0$, there must exist a singularity of negative index. When $g=0$, with $n\geq3$ and $4-4g=4$, it is not feasible for all singularities to have an index of $+2$ simultaneously. Therefore, there must be at least one zero and $P-Q$ non-empty.
\end{proof}

The cases where $2g-2+n\leq0$ are discussed in Section \ref{sec:2cone_sphere}. Unless stated otherwise, we will assume that $2g-2+n>0$ by default.

By Lemma \ref{lem:Sing_G}, each connected component of $\Sigma-\SingG$ is foliated by geodesic leaves of length $\pi$, which form a bigon. Moreover, all the singularities $P\cup Q$ are located on the boundaries of these bigons.

\begin{definition}\label{def:strip_decomp}
    Given a foliated surface $(\Sigma;F,G)$, the critical leaves $\SingG$ partition the surface into a finite number of bigons. This partition is called the \DEF{strip decomposition} of $(\Sigma;F,G)$.
\end{definition}

The original surface $\Sigma$ can be reconstructed by gluing several \EMP{closed} bigons along their boundaries. 
All the vertices of the bigons are glued to the poles of the foliation. Conversely, all points in $Q$ are fromed by the vertices of the bigons. Otherwise if all leaves from some $q\in Q$ do not meet any zero, they must be geodesics of length $\pi$ and terminate at the same pole. These leaves already form a closed football.

A boundary point of a bigon, other than the two vertices, that is glued to a cone point is considered as a \DEF{marked point} of the closed bigon before gluing. 
A cone point in $P-Q$ of angle $s\pi\ (s\in\ZZ_+)$ is glued from $s$ marked points. 
The same convention applies to the annulus decomposition later. 
The \DEF{angle} or the \DEF{width} $w_i$ of a bigon $B_i=B_{w_i}$ (see Section \ref{sec:spherical_cone}) in the strip decomposition equals to intersection number $i(G,\alpha)$, where $\alpha$ is a leaf segment of $F$ inside $B_i$ connecting its two boundaries. This quantity is also denoted as $i(G, B_i)$.

For co-axial surfaces, the strip decomposition is equivalent to the football decomposition in \cite[Theorem 1.3]{WWX22axv}.

\bigskip
The number of bigons, the division of the boundaries of the bigons, and the pairing of boundary segments are called the \DEF{gluing pattern} of the strip decomposition. This is a purely combinatorial data.
The strip decomposition of a foliated surface $(\Sigma;F,G)$ is unique. So the width of the bigons, the length of boundary segments, 
and the gluing pattern completely parameterize all foliated surfaces. However, the choice of $(F,G)$ on a given $\Sigma$ may not be unique, as illustrated in Example \ref{eg:three_cone_sphere} below. Despite this, these parameters are sufficient to distinguish between different dihedral surfaces.

\begin{proposition}
    The strip decomposition provides a complete determination for a dihedral surface. Specifically, two dihedral surfaces $\Sigma_1, \Sigma_2$ are isometric if and only if there exist corresponding foliated surfaces $(\Sigma_i;F_i,G_i), i=1,2$, such that their strip decompositions are identical. This means that they have the same gluing pattern, bigon widths and position of boundary marked points.

    In other words, the equivalence holds if and only if there exists a \homeo\ $I:\Sigma_1\to\Sigma_2$ such that $I^*(F_2)=F_1, I^*(G_2)=G_1$ as \mfoli s.
    \hfill $\square$
\end{proposition}
For labeled surfaces, we require the \homeo\ $I$ to preserve the labels. The gluing pattern is a discrete combinatorial data. Therefore, counting the dimension of the moduli space is equivalent to counting the independent continuous parameters of the strip decomposition.

\begin{example}\label{eg:three_cone_sphere}
    The strip decomposition in Figure \ref{fig:exp_strip_1}, consisting of three bigons of width $\pi$ (equivalent to hemispheres), gives a surface in $\MsphC{0}{3}{\frac12,\frac12,3}$.
    The foliated surface belongs to $\MsphCFT{0}{3}{(\frac12,\frac12,3)}{\vec{T}}$ where $\vec{T}=( \{p_3\}, \emptyset, \{p_1, p_2\} )$. This foliation contains two extra poles $q_1, q_2$.
    \begin{figure}[!h]
      \centering
      \includegraphics[width=0.9\textwidth]{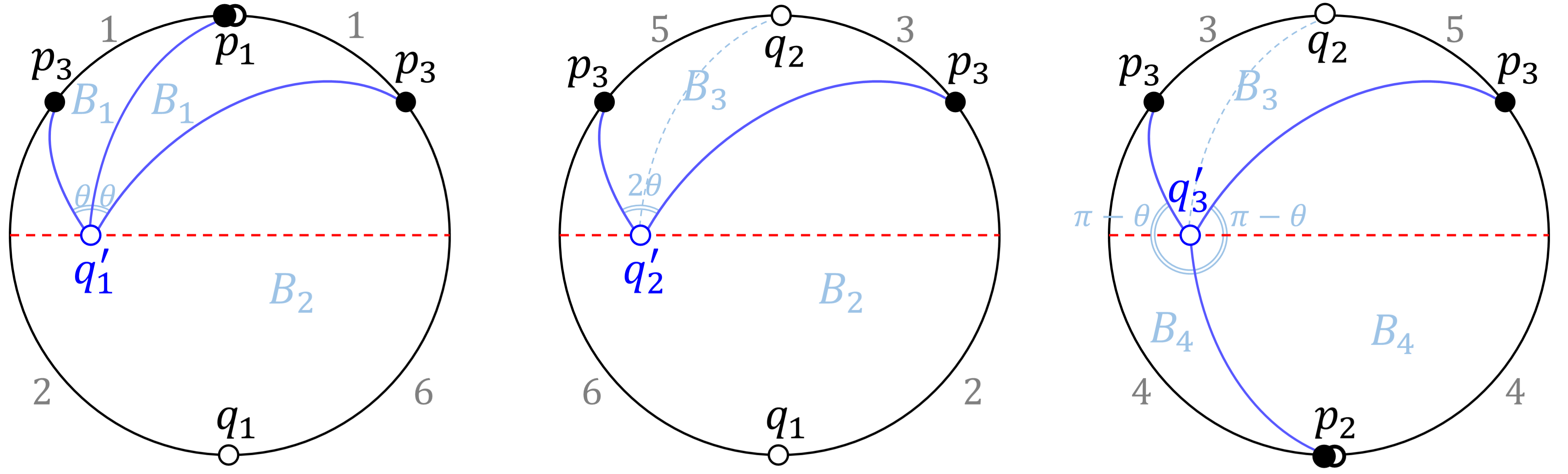}
      \caption{A strip decomposition of a surface in $\MsphC{0}{3}{\frac12,\frac12,3}$ into three hemispheres. The boundary segments labeled with the same number are glued in pair. The lengths of boundary segments of label $1,3,5$ are the same.} 
      \label{fig:exp_strip_1}
    \end{figure}
    
    However, there exits another continuous family of strip decomposition induced by non-orientable foliations. The new poles $q'_1,q'_2,q'_3$ lie on the equatorial net of the previous foliation and are smooth points of the metric. Their positions on the hemispheres are exactly the same. The type partition $\vec{T}'= ( \{p_3\}, \{p_1,p_2\}, \emptyset ) $. The bigons have width of $\theta, 2\pi-2\theta, 2\theta, \pi-\theta$ where $0<\theta<\pi$. 
    \begin{figure}[ht]
      \centering
      \includegraphics[width=0.9\textwidth]{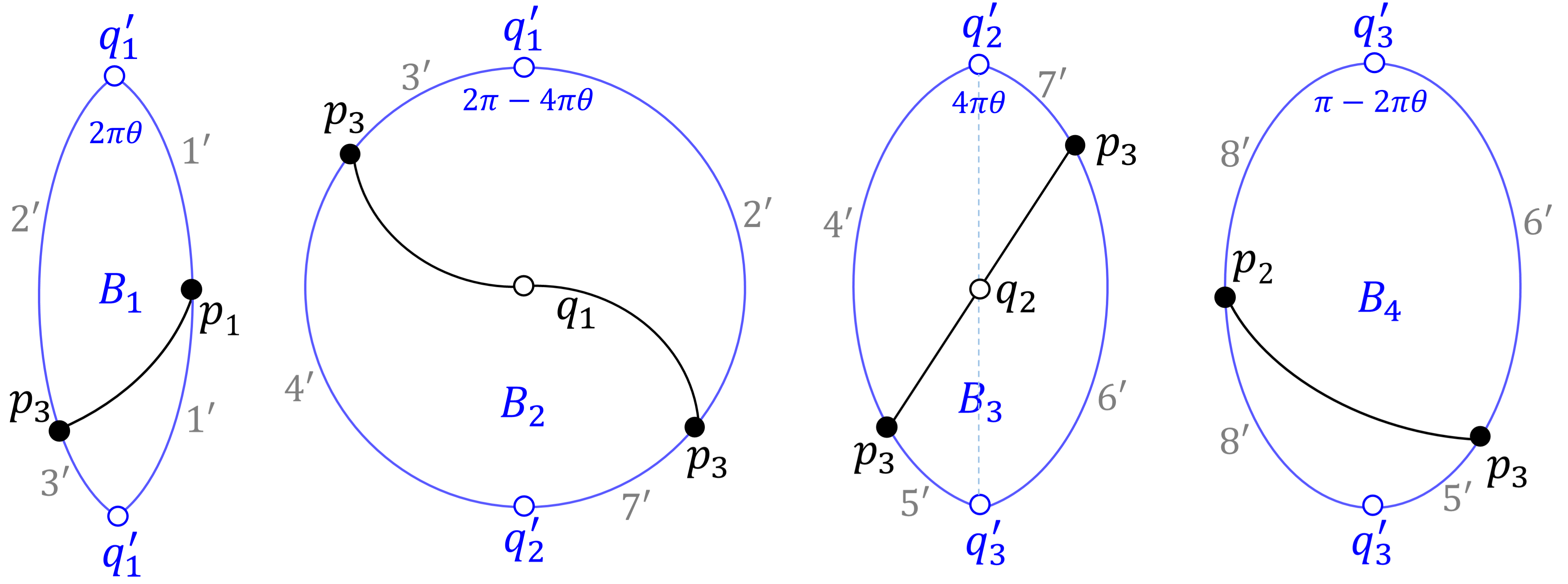}
      \caption{A strip decomposition of the same surface in Figure \ref{fig:exp_strip_1} into four bigons. Two marked points $p_1\in\pt B_1, p_2\in\pt B_4$ are the midpoints. And the lengths of boundary segments of label $3',5',7'$ are the same.}
      \label{fig:exp_strip_2}
    \end{figure}

    In both figures, the critical leaves of the first and second longitude foliation are represented by the segments in black and blue respectively. The representation of singularities follows the figure legend in Table \ref{tab:singularity}. The monodromy groups of the two foliated surfaces are $\rho_1(\pi_1)=\{\mathrm{id}, z\mapsto -z\}$ and $\rho_2(\pi_1)=\{\mathrm{id}, z\mapsto 1/z\}$.
    \hfill $\square$
\end{example}
\textbf{Remark.} This example highlights three important facts:

\textbf{(1).} The orientation requirement in Definition \ref{def:moduli_C}.(2) is necessary. Without such condition, a foliated co-axial surface could potentially be misclassified as a strict dihedral one. Furthermore, the fiber of $p$ in (\ref{eq:moduli_spaces_C}) over such a surface might contain a component with misleading extra dimension. 

\textbf{(2).} The existence of a primitive \qdf\ does not guarantee a strict dihedral metric. It only indicates that the current developing map induces a primitive differential, but this may not hold for all possible developing maps. The monodromy group is the truly essential quantity. 

\textbf{(3).} On the same Riemann surface, it is possible to have both an abelian and a primitive \qdf\ with the same set of zeros and all real periods. However, the poles of these differentials may be distinct. 

\subsection{Annulus decomposition}
There is another decomposition induced by $F$, perpendicular to the previous one. It is exactly the same as \emph{ring domain decomposition} in \cite{Li19} and foliated cylinder decomposition in \cite{Erm20}.

\begin{definition}
    The set of critical leaves of $F$, denoted by $\SingF$, is defined as the union of all points in $Q$ and leaves of $F$ starting from $P-Q$.
\end{definition}

\begin{lemma}
    All leaves in $F-\SingF$ are closed regular leaves. The leaves in the same connected component of $F-\SingF$ are all homotopic. 
\end{lemma}
\begin{proof}
    We can regard the leaves of $F$ as the level set of the Morse function $H^2:\Sigma\to[0,1]$, where $H$ is the absolute latitude foliation defined in Section \ref{sec:ll_foli}. The singularities of $F$ are the critical points of function $H^2$. Then every level set is a closed compact subset of $\Sigma$. Each connected component is a codimension 1 submanifold when it does not contain a critical point. By compactness, it must be a circle.

    For a connected component $A$ of $F-\SingF$, the restriction of $H^2$ on $A$ is a Morse function with circle fibers over its image. If $0$ is not in the image, then $A$ is a circle bundle over an interval, hence an annulus. If $0$ is in the image, then the preimage circle of $0$ cut $A$ into two annulus in the previous. Since there is no singularity, $A$ is glued from two annuli along one-side boundary, hence an annulus again.

    This result can also be proved by the strip decomposition. Each leave from a zero of $F$ must end at a zero with the same absolute latitude, because the whole level set is compact. By drawing all these leaves in $F$, each bigon is further divided into rectangles bounded by a pair of geodesic leaf segments in $G$ and a pair of leaf segments in $F$ (the later one may degenerate to a pole point). All remained leaves of $F$ are regular and closed. So gluing these pieces along the segments of $G$, one must obtain punctured disks and annuli.
\end{proof}

\begin{definition}\label{def:ann_decomp}
    Given a foliated surface $(\Sigma;F,G)$, the critical leaves $\SingF$ partition the surface into a finite number of annuli and punctured disks. This partition is called the \DEF{annulus decomposition} of $(\Sigma;F,G)$.
\end{definition}

As before, $\Sigma$ is obtained by gluing the closed annuli and disks along their boundaries. For a component $A_j$ of the annulus decomposition, the total latitude variation $h_j$ is defined as its \DEF{height}. It equals to intersection number $i(F,\beta)$, where $\beta$ is any leaf segment of $G$ inside $B_i$ connecting its two boundaries. This quantity is also denoted as $i(F, A_j)$.
The total longitude variation $c_j$ of $A_j$ is defined as its \DEF{circumference}. It equals to intersection number $i(F,\gamma)$, where $\gamma$ is any closed regular leaf of $F$ inside $A_j$. This quantity is also denoted as $i(G, A_j)$.

Figure \ref{fig:decomp} depicts the strip and annulus decomposition of a genus zero foliated surface with 6 cone points. $p_1, p_2$ are the zeros of the foliation, with cone angle $4\pi$. The remaining 4 points $q_1, q_2, q_3, q_4$ are the poles of the foliation. Their cone angles are determined by the widths of the bigons in strip decomposition.
\begin{figure}[htb]
  \centering
  \includegraphics[width=0.9\textwidth]{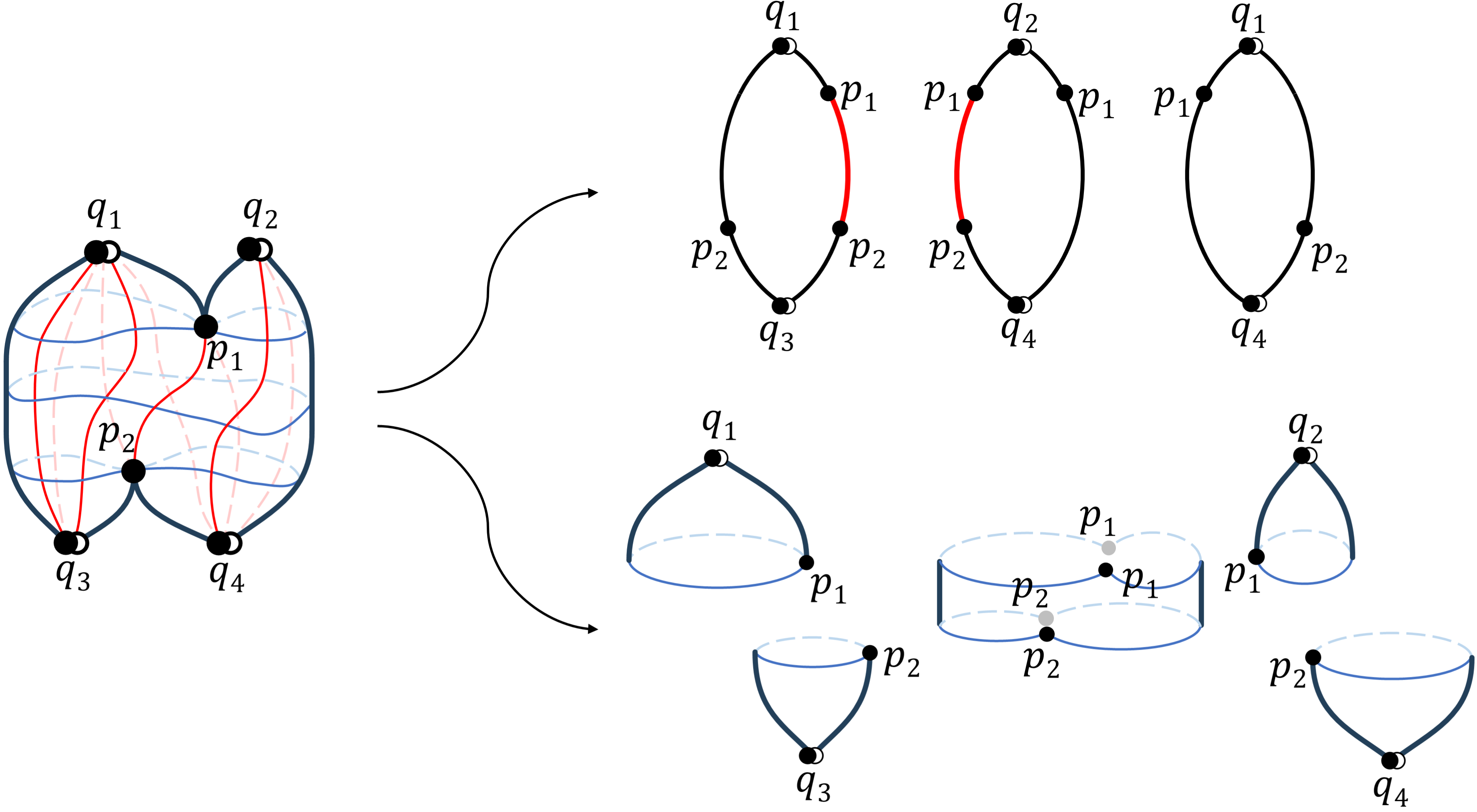}
  \caption{The strip decomposition (right top row) and the annulus decomposition (right bottom row) of a foliated dihedral surface.}
  \label{fig:decomp}
\end{figure}

\subsection{Sliding deformation}
By utilizing latitude and longitude foliations, we are able to define classes of geometric deformations for dihedral surfaces. 

\begin{definition}
    Let $(\Sigma_1; F_1, G_1)$, $(\Sigma_2; F_2, G_2)$ be two foliated dihedral surfaces with the same topological type and angle vector. They are said to \DEF{be differed by a sliding deformation} if there exists a homeomorphism $J:\Sigma_1\to\Sigma_2$ preserving the label, such that $G_1=J^*G_2$ as \mfoli s on topological surfaces.

    In other words, a \DEF{sliding deformation} of a foliated dihedral surface is changing the length of boundary segments while keeping the gluing pattern and widths of the bigons in strip decomposition.
\end{definition}

The effect of sliding is changing the position of the zeros of $G$. Note that cones of angle $(2k-1)\pi$ must lie on the equatorial net, so sliding can only change the position of cone points of angle $2k\pi$, where $k\in\ZZ_{>1}$.
The angle vector is kept during sliding, and the underlying Riemann surface is usually changed.

\bigskip
This deformation is described geometrically and can be applied locally. It also includes a classes of global deformation discovered before. 
It is observed in \cite{CWWX15} that an abelian differential $\omega$ 
induces a 1-parameter family of co-axial metrics, expressed by the developing map
$$ D_\lambda(z) := \lambda \cdot \exp\left( \int_{z_0}^z \omega \right) := \lambda \cdot D_1(z),\quad \lambda\in\RR_+ $$
as a multi-valued meromorphic function on $\Sigma-P$. We call changing the positive parameter the \DEF{$\lambda$-deformation} of a co-axial surface. Note that the underlying Riemann surface and the angle vector are fixed and $D_\lambda(z)$ is always conformal.

The $\lambda$-deformation changes the radial distance from the origin while keep the argument. Thus the longitude foliation is kept during the deformation, and so does the gluing pattern of strip decomposition. 
If $\abs{D_1(z)}=r(z)$, then the spherical distance from $D_1(z)$ to $0$ is $\int_0^{r(z)}\frac{2\mathrm{d}w}{1+w^2}=2\arctan r(z)$. Thus the spherical distance from $D_\lambda(z)$ to $0$ is $2\arctan\abs{D_\lambda(z)}=2\arctan(\lambda\cdot r(z))$. It is just a reassignment of the length of boundary segments of all bigons, thus a sliding deformation. See Figure \ref{fig:sliding_lambda}. 

\begin{figure}[ht]
  \centering
  \includegraphics[width=\textwidth]{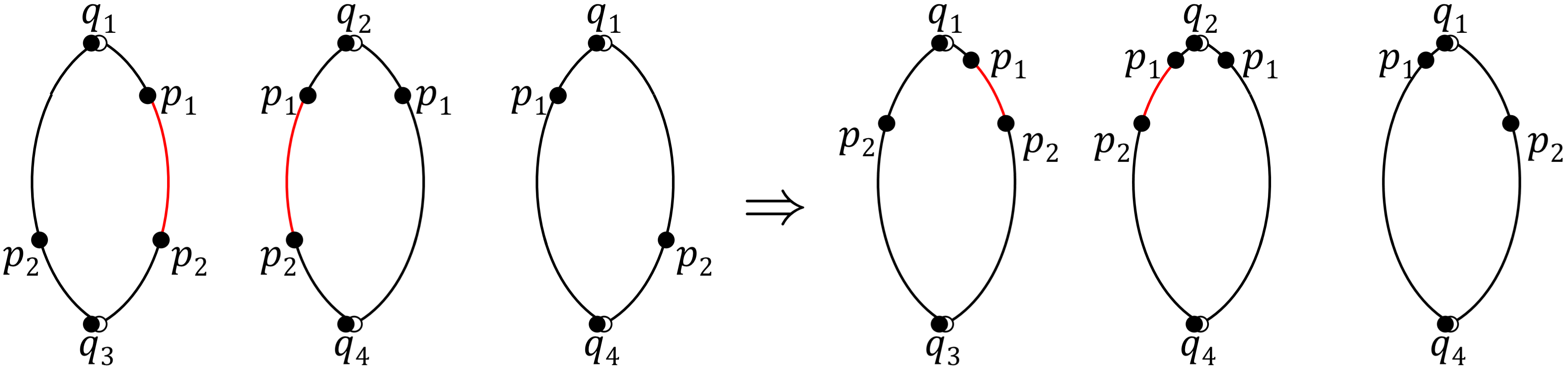}
  \caption{The effect of $\lambda$-deformation, which is a special case of sliding deformation.}
  \label{fig:sliding_lambda}
\end{figure}

\begin{remark}\label{rmk:equi_class}
    The 1-forms $D^*_\lambda(\ii z/\mathrm{d}z)$ on $\Sigma$ are the same for all $\lambda>0$. Thus one can not distinguish between the 1-parameter family of co-axial metrics that differ by $\lambda$-deformation based solely on 1-forms on the surface. These families of metrics are usually considered as one equivalent class in literature, such as in \cite{EGT14,CLY22axv}. Our modification on the transverse measure of $F$ allows us to distinguish this family by encoding more geometric information. Therefore, the moduli space of foliated surfaces is finer than the moduli space of 1-forms that induce co-axial metrics.
\end{remark}

\begin{remark}
We can consider the limit surface as $\lambda$ approaches 0. The distance from any points to the origin tends to $0$, except for $\infty$. This means that we lost information about all cone points, except for those mapped to $+\infty$. In terms of the strip decomposition, what remains is solely the number and width of bigons, and the gluing pattern for boundary segments that end at the south vertices. 

Based on this observation, we can assert that the limit of $\lambda$-deformation of a co-axial surface is a singular surface glued from several footballs along one of their vertices. In the summer of 2023, Zhiqiang Wei also achieved identical results for co-axial spheres.
The number of footballs in the limit surface equals to the number of preimages $D^{-1}(\infty)$ on $\Sigma$. The case of $\lambda\to+\infty$ is similar. Such degenerated surfaces merit further investigation and are related to the boundary of moduli space.
\end{remark}

\subsection{Splitting a pole}\label{sec:split}
This deformation turns a pole of the foliation of angle $2k\pi$ into a zero of the same angle. 
It can be regarded as a special kind of sliding deformation.

\begin{definition}\label{def:split}
    For $p\in P\cap Q$ of cone angle $2k\pi, k\in\ZZ_{>1}$, 
    the \DEF{split deformation} slightly move it along $G$, so that the pole splits into $k$ poles which are smooth point on the surface, and the cone point becomes a zero of foliation with index $2-2k$. See Figure \ref{fig:split_pole} and the geometric description below.
    \begin{figure}[ht]
      \makebox[\textwidth][c]{\includegraphics[width=1.05\textwidth]{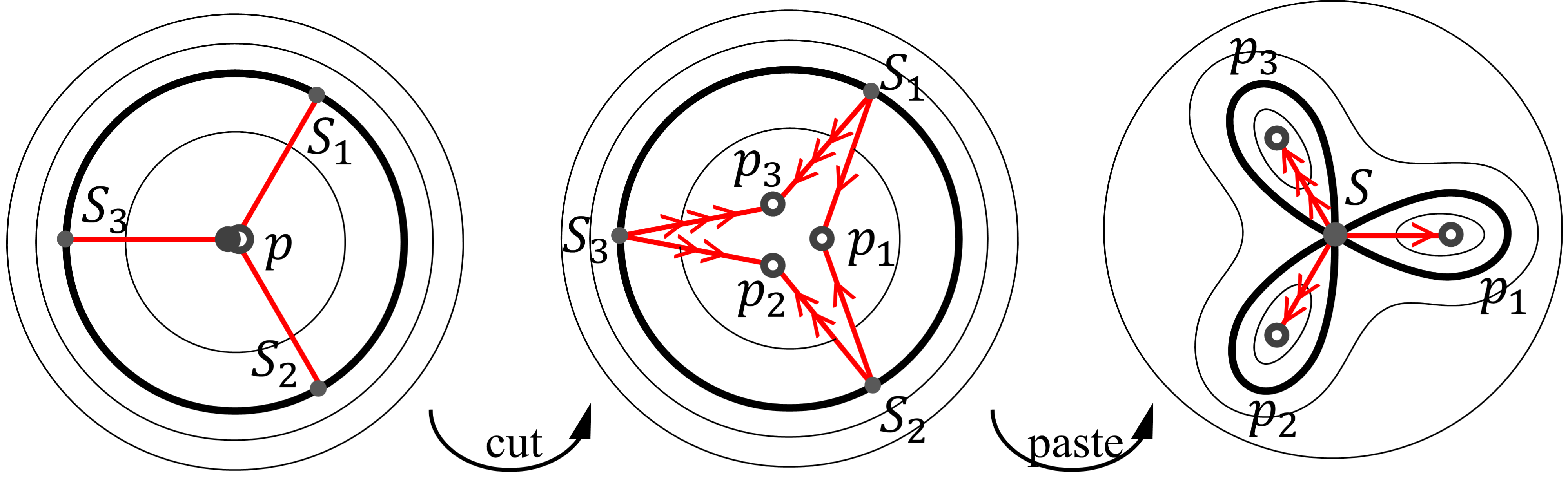}}
      \caption{Splitting a pole of cone angle $6\pi$. Some leaves of $F$ is shown. Pay attention to the bold line that switches from a regular leaf to singular leaves.}
      \label{fig:split_pole}
\end{figure}
\end{definition}
Geometrically, this deformation is obtained by a cut-and-paste operation. First pick $k$ isometric longitude segments $v_1=p S_1, \cdots, v_k=p S_k$ from $p$, equally spaced by $2\pi$ angle. The length $\ell$ of these closed segments is small enough so that they contain no singularity other than $p$. 
Cutting the surface along $v_i$'s, one obtain a bordered surface with one boundary component. The cone point $p$ splits into $k$ boundary vertices $p_1, \cdots p_k$. Suppose the vertices are labeled as $S_1, p_1, S_2, p_2, \cdots S_k, p_k, S_1$ in counterclockwise order. Then glue $p_i S_i$ to $p_i S_{i+1}$ for all $i=1,\cdots, k$ (indices are modulo $k$). The foliation structure matches naturally, because the boundary are leaf segments of $G$ and perpendicular to leaves of $F$. All $S_i$ are glued to a common point, which is the new cone point of angle $2k\pi$, and each $p_i$ becomes a pole. 

There is a 2-dimensional choice when splitting a zero: the directions of the $k$ equally spaced segments, and the length $\ell$ of the segments.
This operation also appears in \cite{Erm20}.

\begin{remark}
    Actually this can be generalized to a local operation that splits a cone angle of $2\pi\theta$ into $k+1$ cone points $(k\geq2)$ of angle $2\pi\theta_1, \cdots, 2\pi\theta_k, 2k\pi$ with $\theta_1+\cdots +\theta_k = \theta$. This does not require dihedral condition, since the foliations can always be defined locally around a cone point.
\end{remark}


The splitting deformation provides some relations between the moduli spaces of different type partitions. 
\begin{proposition}\label{prop:low_dim_bd}
    Given a angle vector $\vec{\theta}:=(\theta_1, \cdots, \theta_n) \in \RR^n_+$ and a type partition $\vec{T}=( E, O, N )$, assume that the moduli space $\MsphgnDFT$ is non-empty. If there is some $k\in N$ such that $\theta_k\in\ZZ$, then $\MsphgnDFT$ appears at the lower dimensional boundary of some other space $\MsphDF{g}{n}{\vec{\theta},\vec{T}^0}$. Here $\vec{T}^0$ is obtained by moving every $k\in N$ with $\theta_k \in \ZZ$ into $E$. Similar results hold for co-axial surfaces.
\end{proposition}

\begin{proof}
    If $(\Sigma;F,G)\in \MsphgnDFT$, for each $k\in N$ with $\theta_k\in \ZZ$, we may apply a splitting deformation at cone point $p_k$, resulted in a foliated surface in $\MsphDF{g}{n}{\vec{\theta},\vec{T}''}$ with $\vec{T}''= ( E\cup\{ k \}, O, N-\{ k \} )$. As pointed out above, the choice of splitting has a 2-dimensional parameter space. Viewing this deformation in reverse order, the original foliated surface can be regarded as a limit of surface in $\MsphDF{g}{n}{\vec{\theta},\vec{T}''}$ with some marked points on the boundary tending to the vertex of the bigons. 
    Due to the choice of splitting, this limiting set has lower dimension. And the propersition is obtained by induction.
\end{proof}

\subsection{Twist deformation}
Another type of deformation is “perpendicular” to the previous one. Specifically, this deformation preserves the latitude foliation. 
It resembles the “twist” in Fenchel-Nielsen coordinates for hyperbolic surfaces. 
A flat version is also discussed in the literature, for example, see \cite[Section 3.1]{Gupt15}.

Let $(\Sigma; F,G)$ be a foliated surface. 
Let $B$ be a component of its annulus decomposition, other than a punctured disk, with circumference $c=i(G,B)$. Then $B$ is foliated by closed leaves of $F$. Pick any such closed leaf $\gamma$ in its interior, and let $f_\gamma:[0,c]\to \gamma\subset B$ be a parameterization with respect to the transverse measure of $G$, where $f_\gamma(0)=f_\gamma(c)$. Cut $B$ into two half annuli $B^+, B^-$ along $\gamma$, with $B^+$ on the left of $\gamma$. The cut locus of $\gamma$ as the boundary of $B^\pm$ is denoted by $\gamma^\pm$, parameterized by $f_\gamma^\pm:[0,c]\to \gamma^\pm$ with $f_\gamma^\pm(0)=f_\gamma^\pm(c)$ as before. Then $ B = B^+ \cup B^- \big{/} f_\gamma^+(x) \sim f_\gamma^-(x),\ \forall x\in[0,c] $.

\begin{definition}\label{def:twist}
    Let $\psi\in[0,c)$. A \DEF{$\psi$-longitude (right) twist deformation} of $(\Sigma;F,G)$ along $B$ is cutting $\Sigma$ along $\gamma$ and re-gluing $\gamma^+,\gamma^-$ by the identification
    \[ f_\gamma^+([x]_c) \sim f_\gamma^-([x+\psi]_c) \]
    for all $x\in [0,c]$. Here for any $t\in\RR$, $[t]_c$ is the real number in $[0,c)$ such that $[t]_c \equiv t (\mathrm{mod}\ c)$.
\end{definition}

\begin{figure}[ht]
  \centering
  \includegraphics[width=0.85\textwidth]{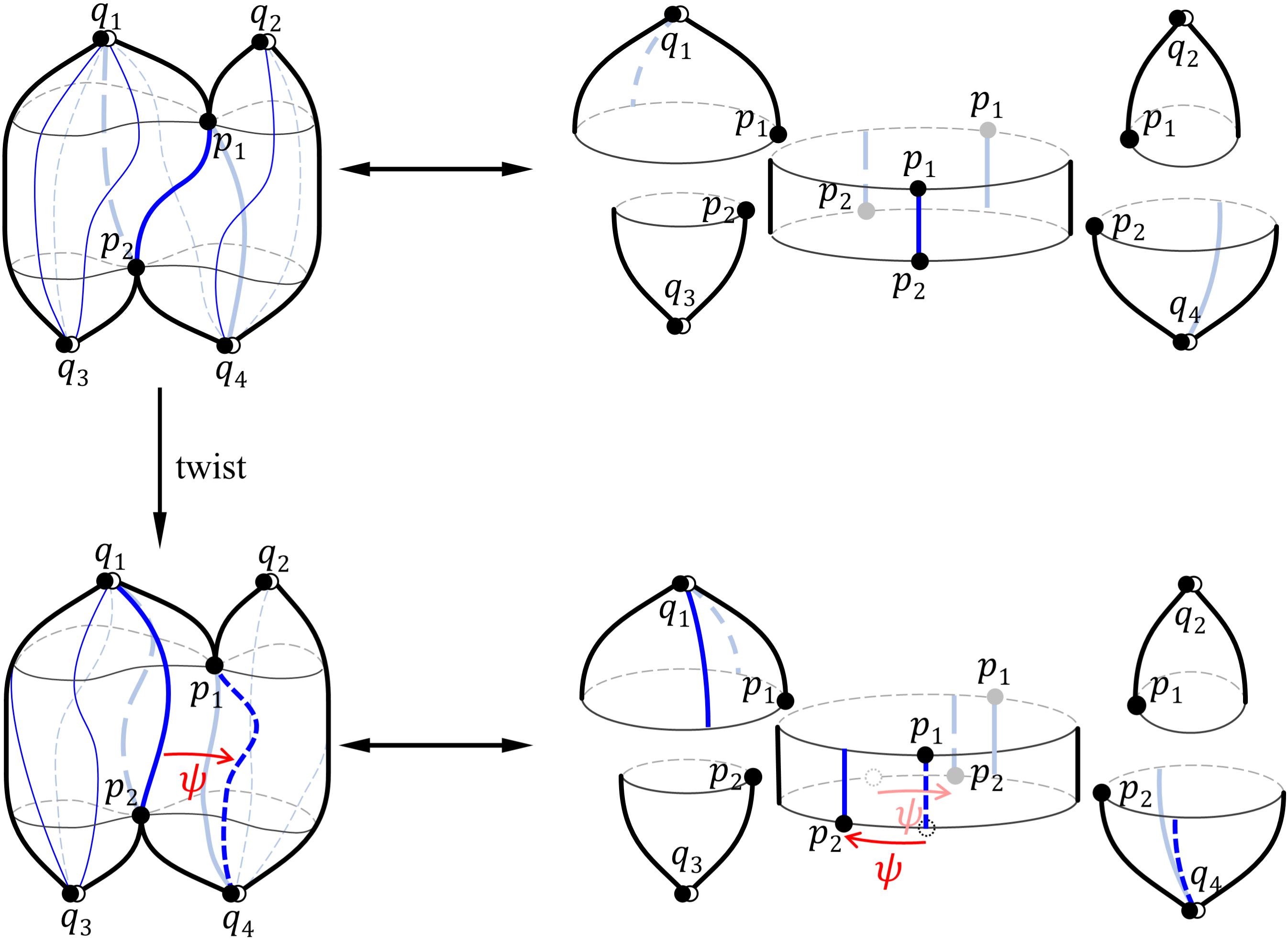}
  \caption{A positive twist is applied along the central annulus, with the longitude foliations being depicted. This causes a change in the strip decomposition. It appears that $G$ undergoes a leftward twist, because a leaf segment now connects to a segment that was previously positioned to the left.}
  \label{fig:twist}
\end{figure}

Intuitively, the twist deformation alters the alignment of the leaves of $G$ when they traverse the annulus $B$. Observed from one side of $B$, everything on the opposite side appears to shift to the right. Note that $F$ remains unchanged throughout this process.
The deformation is independent of the selection or orientation of $\gamma\subset B$. The parameter $\psi$ can take any real value. Nevertheless, considering our focus on the moduli space (rather than the \teich\ space), the result remains identical if the twist parameters are congruent modulo $c$. Consequently, twist deformation manifests as an $\SOne$-action on a continuous family of dihedral surfaces.

\begin{remark}
    Twist deformations can be applied to general spherical surfaces, extending beyond just the dihedral ones. The crucial element is a simple closed curve with constant geodesic curvature and not being homotopic to a cone point. 
\end{remark}

\subsection{Generic foliated surface}\label{sec:generic}
As an application of the previously defined deformations, we can discuss the typical characteristics of most dihedral surfaces. These surfaces have the highest number of independent parameters, thereby constituting the top-dimensional subset within the moduli space. Our attention will be focused on these surfaces for dimension counting.
\begin{definition}\label{def:generic_foliated}
        A foliated surface $(\Sigma;F,G)$ is said to be \DEF{generic} if each boundary of the bigons in strip decomposition contains exactly one marked point.
        This is equivalent to the absence of a singular leaf connecting zeroes of $G$.
\end{definition}

We now introduce a reversible deformation process that transforms a non-generic surface into a generic one. This allows us to view a non-generic surface as a special case where certain bigons in the strip decomposition possess zero width. This explains the usage of the term “generic”.

\begin{proposition}\label{prop:generic}
    Each non-generic foliated surface $(\Sigma;F,G)$ can be continuously deformed to a generic one with fixed cone angles and type partition.
\end{proposition}
\begin{proof}  
    We call a critical leaf connecting zeros of $G$ \DEF{vertical saddle connection}. Since $P$ is finite, there are finitely many vertical saddle connections in $\SingG$. Let $\sigma$ be such a leaf connecting two zeros $p_1, p_2$, possibly the same.

    Now in the annulus decomposition of $(\Sigma;F,G)$, there exists an annulus $A=A(\sigma)$ passing through $\sigma$, and containing $p_1$ in one boundary component $C_1\subset \pt A$. It must not be a punctured disk. Otherwise the other end of $\sigma$ is the pole in that disk. 

    There are finitely many leaves in $\SingG$ intersecting $A$. These leaves, ending in singularities of $G$, must hit $C_1$. Then $C_1$ is divided into several open segments $I_1, \cdots, I_K$ by these intersections. Denote longitude variation of $I_j$ by $d_j=i(I_j,G)>0$. Let $t$ be a real number with $0<\abs{t}<\min\{d_1,\cdots,d_K\}$. 
    Then a $t$-longitude twist along $A$ will break the vertical saddle connection $\sigma$, and not produce new vertical saddle connection. See Figure \ref{fig:twist_generic}. So we find a 1-parameter deformation of $(\Sigma;F,G)$ with fewer vertical saddle connections. Everything is unchanged outside the annulus $A$, and the local gluing data around cone points on $\pt A$ is fixed. So the angle vector and the type partition of the new surfaces are all the same.

    Finally, by induction, one can break all vertical saddle connections without changing the cone angles and the type partition. In addition, the absolute latitude of all cone points are fixed.
\end{proof}

\begin{figure}[ht]
  \makebox[\textwidth][c]
  {\includegraphics[width=\textwidth]{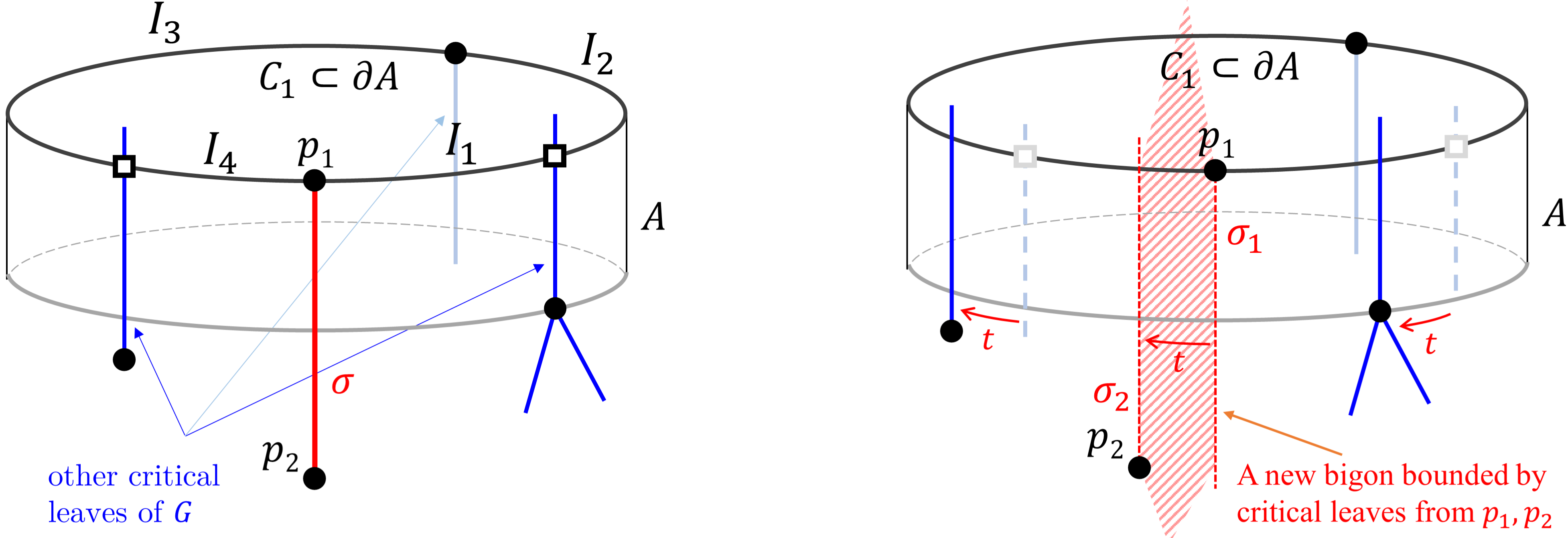}}
  \caption{Twist along an annulus $A$ to break vertical saddle connections $\sigma$. A new bigon appears after the deformation.}
  \label{fig:twist_generic}
\end{figure}

\begin{corollary}\label{cor:low_dim}
    Non-generic foliated surfaces form a lower dimensional subset in $\MsphgnDFT$ or $\MsphgnCFT$ for every possible type partition $\vec{T}$.
\end{corollary}
\begin{proof}
    Follow the notations above. After twist deformation, the critical leaves from $p_1$ and $p_2$ in the direction of $\sigma$ mismatch at longitude variation $t$. By the choice of $t$, no other critical leaf lies in between them. So these two critical leaves lie on the boundary of a new bigon in strip decomposition.

    View the deformation in reverse order, the original foliated surface is obtained by shrinking the width of a certain bigon to zero.
    The more times of twist deformation applied to obtain a generic surface, the more bigons need to be shrunk. Note that the absolute latitude of each cone point remains unchanged during the process. So non-generic surfaces have fewer parameters than generic ones.
\end{proof}


\subsection{Genus zero surface with two cone points}\label{sec:2cone_sphere}
For completeness, we now discuss the cases where $2g-2+n\leq0$. The only choices for such $(g,n)$ are $(0,0),(0,1),(0,2),(1,0)$. 

By Gauss-Bonnet formula, a sphere with one cone point must be smooth, and torus without cone point can not be spherical. 
The only remaining case is $(g,n)=(0,2)$. However, genus zero spherical surfaces with two cone points has already been studied \cite{Tro91}. 
Also see the description in \cite[Theorem 2.10]{Tah22}. 
We present a topological proof here.

\begin{proposition}
    If $\Sigma$ is a genus 0 cone spherical surface with 2 cone points, other than the football $S_\theta$ (See Section \ref{sec:spherical_cone}), then it is obtained by cyclically gluing $m\in\ZZ_+$ copies of unit sphere with a same slit of length $l\in(0,\pi)$ as Figure \ref{fig:SlitSph}.
\end{proposition}
\begin{figure}
    \centering
    \includegraphics[width=0.8\linewidth]{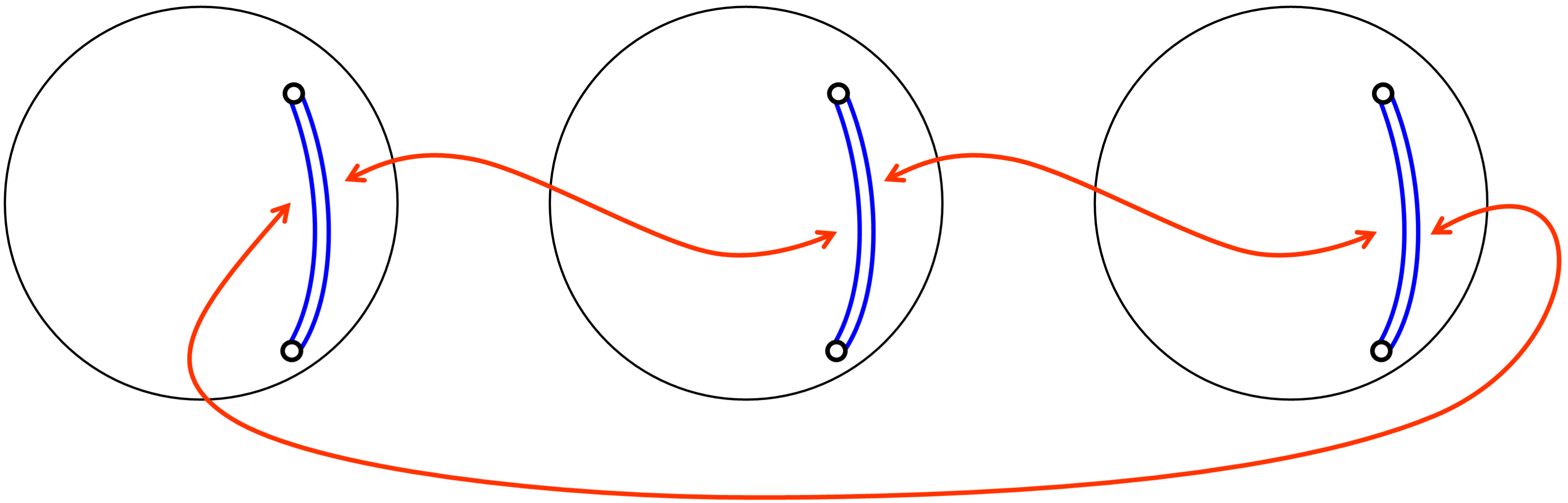}
    \caption{A genus zero surface with two cone points of angle $6\pi$, which is obtained by cyclically gluing slit spheres.}
    \label{fig:SlitSph}
\end{figure}

\begin{proof}
    Let $P=\{p_1, p_2\}$. Since $\pi_1(\Sigma-P)\cong \ZZ$, $\Sigma$ must be co-axial. Let $(\Sigma;F,G)$ be a foliated surface induced by developing map $D$ such that $\rho(\pi_1)\subset O(1)$. Recall that $Q=\pi D^{-1}\{N,S\}$.
    By \PH\ formula (\ref{eq:PH}), $\sum \ind(p) = 2(2-2g) = 4$. Since every cone point is a singularity of $G$, there is at most 2 zeros of $G$, or $\#(P-Q)\leq2$.

    If $p_1,p_2\in Q$, there is no other singularity, and $\SingG = \emptyset$. Then $\Sigma$ must be some football $S_\theta$.

    If $p_1,p_2\in P-Q$, then $\theta_1,\theta_2\in\ZZ$ and $\rho(\pi_1)$ is trivial. We can always compose a $\SOthr$ element so that $D(\pi^{-1}(p_1))=N$. So we only need to consider the case where $p_1\in Q$ and $p_2\in P-Q$. Then $\theta_2 = m \in\ZZ_+$ and $\rho(\pi_1)$ is trivial. So $D$ reduced to a univalent map, also denoted by $D$, directly from $\Sigma$ to $\STwo$. Let $X=D(p_2)\notin \{N,S\}$ and $B=D^{-1}(X,N)$. The developing map $D$ is a local \homeo\ outside cone points, so the restricted map $D:\Sigma-B \to \STwo-\{N,X\}$ is a covering map. 
    Since $\theta_2 = m$, the degree of cover is $m$. Thus the reduced map $D:\Sigma\to\STwo$ is a $m$-sheet covering branched at $p_1,p_2$, and $\theta_1=\theta_2=m$. We obtained the desired result.
\end{proof}

\begin{theorem}
    When $\theta_1\neq\theta_2$, $\Msph{0}{2}{\theta_1,\theta_2}$ is empty. For $\theta\notin\ZZ$, $\Msph{0}{2}{\theta,\theta}$ is a single point. For $\theta\in\ZZ_{>1}$, $\Msph{0}{2}{\theta,\theta}$ is isomorphic to the interval $(0,\pi/2]$.
\end{theorem}

\begin{proof}
    The first two assertion is a direct corollary of the previous proposition.
    When $\theta_1=\theta_2=\theta\in\ZZ$, following the notations above, the only real parameter for the branched covering is the distance from $X=D(p_2)$ to $N$, which lies in $(0,\pi/2)$. Also note that when that distance is $\pi/2$, it is just the football $S_{\theta}$. 
\end{proof}

\textbf{Remark.} Strip decomposition and techniques in next section still work when $\theta\in\ZZ$, since $\SingG\neq\emptyset$. But most gluing patterns are not feasible. 
Topological method is employed to obtain a global perspective.

\section{Dimension counting}\label{sec:Dimension}
To proof Theorem \ref{thm:Main}, we chase diagram (\ref{eq:moduli_spaces_D}) and (\ref{eq:moduli_spaces_C}) from the upper left corner. We first study the fiber of projection $p$. 
Then we count independent parameters for foliated surfaces of a given type partition. Finally, we vary the type partition and obtain the dimension count. The overall idea is summarized in the flow chart of Figure \ref{Summary} at the end of this section. 
Strict dihedral and co-axial cases are treated separately. $2g-2+n>0$ is assumed by default.

\subsection{Identification lemma}
\label{subsec:iden}
Basically, the identification lemma answers how many strip decomposition can be identified to a single surface.
Let us examine the following intriguing example, which highlights the necessity of such lemma.

\begin{example}\label{eg:one_cone_torus}
   A generic foliated surface in $\MsphDF{1}{1}{2}$ is glued by 2 bigons. The width $w_1, w_2$ of the two bigons satisfying $w_1+w_2=\pi$ so that the vertices of the bigons glued to a smooth point. The gluing pattern is shown in the left column of Figure \ref{fig:three_decomp}.

    However, there are two more choices of strip decomposition or foliation. The symmetry center of each bigon is the new pole of another foliation. The critical leaves of that longitude foliation on each bigon are also symmetric about its center. The middle and right column of Figure \ref{fig:three_decomp} shows how to cut the original bigons and glue to new ones. All blue and green arcs are geodesic segments. 

    Let $x,x',x''$ be the length of the shorter boundary segment of bigons in each decomposition, and $w_1, w'_1, w''_1$ be the width of the bigon whose shorter boundary segment lies on the left. By spherical geometry, they satisfy the equations
    \[ \left\{ \begin{array}{lll}
         \cos x' &= \sin x  &\cos(w_1/2)\\
         \cos x  &= \sin x' &\sin(w'_1/2)
    \end{array}\right.,\quad
    \left\{ \begin{array}{lll}
         \cos x'' &= \sin x  &\sin(w_1/2)\\
         \cos x   &= \sin x''&\cos(w''_1/2)
    \end{array}\right. .\]
    \hfill $\square$
    \begin{figure}[hbt]
      \makebox[\textwidth][c]{\includegraphics[width=1.05\textwidth]{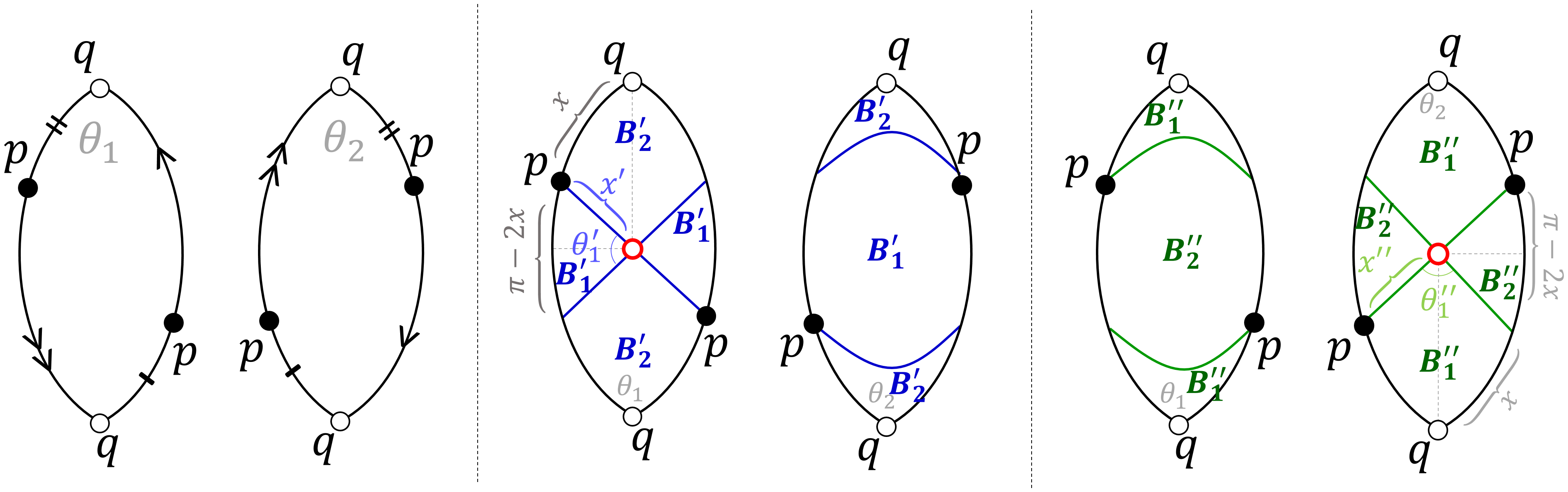}}
      \caption{Three strip decomposition of a same spherical torus with one $4\pi$ cone point.}
      \label{fig:three_decomp}
    \end{figure}
\end{example}

\textbf{Remark.} In fact $\mathcal{M}sph_{1,1}(2)=\MsphD{1}{1}{2}$. Also see \cite[Theorem F.]{EMP20}.

\medskip
Before stating and proving the main proposition we make some settings. Suppose $\Sigma\in\MsphgnD$ admits more than one foliated surfaces, and let $(\Sigma;F,G)$ be one of them in $p^{-1}(\Sigma)\subset\MsphgnDF$ induced by developing map $D$, with monodromy group $\Gamma:=\rho_D(\pi_1)\subset \Otwo$ (see Section \ref{sec:dihedral}).
Every other developing map must be differed from $D$ by a $\PSUtwo$ element, and the monodromy group is differed by a conjugation.
Thus finding another foliated surface is equivalent to finding some $g\in\PSUtwo$ such that $g\Gamma g^{-1} \subset \Otwo$.
Let 
$$\SYM{G_\Gamma}:=\{\ \phi\in\PSUtwo \ \rvert\ \phi^{-1}\Gamma\phi\subset \Otwo \ \}.$$
Obviously, $\Otwo\subset G_\Gamma$. However, $G_\Gamma$ may not be a group. If $\phi,\psi\in G_\Gamma$, then $(\phi\psi)^{-1}\Gamma(\phi\psi)=\psi^{-1}(\phi^{-1}\Gamma\phi)\psi$. Thus $\phi\psi\in G_\Gamma$ 
if and only if $\psi\in G_{\phi^{-1}\Gamma\phi}$.

On the other hand, if $\phi\in G_\Gamma$, then for any $r\in\Otwo$, $\phi\circ r \in G_\Gamma$. So we need to study the right quotient $G_\Gamma/\Otwo$.

\begin{proposition}[Identification lemma for strict dihedral surface]\label{prop:id_SDhd} ~\\
    For strict dihedral $(\Sigma;F,G)$, $\#\{ G_\Gamma/\Otwo \}$ is 1 or 3. And $\#\{ G_\Gamma/\Otwo \}=3$ if and only if $\Gamma$ is isomorphic to the Klein-4 group $K4:= \ZZ_2 \times \ZZ_2$.

    In other words, $\#\{p^{-1}(\Sigma)\} \leq 3$ for all $\Sigma\in\MsphgnD$.
\end{proposition}
\begin{proof}
    Suppose $G_\Gamma\neq \Otwo$.

\textbf{(Step 1).} We pick a representative for an element in $G_\Gamma/\Otwo$.

If $\phi:z\mapsto \frac{az+b}{-\bar{b}z+\bar{a}} $, then
\begin{align*}
    [\phi]:=\phi\cdot\Otwo &= \left\{\ z\mapsto \frac{a\mathrm{e}^{\ii\theta}z+b}{-\bar{b}\mathrm{e}^{\ii\theta}z+\bar{a}},\quad z\mapsto \frac{a\mathrm{e}^{\ii\theta}+bz}{-\bar{b}\mathrm{e}^{\ii\theta}+\bar{a}z} \ \bigg\rvert\ \theta\in\RR \ \right\} \\
    & = \left\{\ z\mapsto \frac{a\mathrm{e}^{\ii\theta/2}z+b\mathrm{e}^{-\ii\theta/2}}{-\bar{b}\mathrm{e}^{\ii\theta/2}z+\bar{a}\mathrm{e}^{-\ii\theta/2}},\quad z\mapsto \frac{b\mathrm{e}^{\ii(\pi-\theta)/2}z+a\mathrm{e}^{\ii(\pi+\theta)/2}}{\bar{a}\mathrm{e}^{\ii(\pi-\theta)/2}z-\bar{b}\mathrm{e}^{\ii(\pi+\theta)/2}} \ \bigg\rvert\ \theta\in\RR \ \right\}.
\end{align*}
Thus $\pm ab$ is the invariant of an element in $G_\Gamma/\Otwo$ in terms of $\PSUtwo$ representation. And $\phi\in\Otwo$ if and only if $ab=0$. So whenever $\phi\notin\Otwo$, we can choose a representative of $[\phi]$ such that $a>0$ and $\arg(b)\in[0,\pi)$.

\medskip
\textbf{(Step 2).} $\Gamma$ keeps a pair of antipodal points other than north and south poles.

Pick non-trivial $[\phi]\in G_\Gamma/\Otwo$, represented by $\phi:z\mapsto \frac{az+b}{-\bar{b}z+\bar{a}} $ with $a>0$ and $\arg(b)\in[0,\pi)$. Then $\phi\{0,\infty\} = \{b/\bar{a}, -a/\bar{b}\} \neq\{0,\infty\}$. Thus
\begin{align*}
    \Gamma\{b/\bar{a}, -a/\bar{b}\} = \Gamma\phi\{0,\infty\}    = \phi\big( \phi^{-1}\Gamma\phi\{0,\infty\} \big)
    = \phi\{0,\infty\} = \{b/\bar{a}, -a/\bar{b}\} .
\end{align*}
So $\Gamma$ keeps two different pairs of antipodal points.

\medskip
\textbf{(Step 3).} The rotation subgroup $\Gamma_0:=\Gamma\cap \Uone$ must be $\ZZ_2$, so $\Gamma\cong K4$.

$\Gamma_0$ can not be trivial. Otherwise for any $\gamma, \gamma'\in \Gamma-\Uone$, $\gamma\gamma'\in \Gamma\cap\Uone=\{\mathrm{id}\}$. However, every $\gamma\in \Gamma-\Uone$ is an order-2 element, or $\gamma^2=\mathrm{id}$. Thus there is only one element in $\Gamma-\Uone$ and $\Gamma\cong\ZZ_2$, contradicting strict dihedral condition.

Now let $\psi:z\mapsto \mathrm{e}^{\ii\theta}z$ be a non-trivial element in $\Gamma_0$ with $\ \theta\in\RR-2\pi\ZZ $. Then $\psi$ should fix the antipodal points $\{b/\bar{a}, -a/\bar{b}\} \neq\{0,\infty\}$. This implies $\psi=\gamma_0:=z\mapsto -z$. Then $\Gamma_0=\{\mathrm{id},\gamma_0\}$, and $\Gamma$ has 4 elements. Since all elements in $\Gamma-\Gamma_0$ are of order 2, $\Gamma\cong K4$. We have
$$ \Gamma=\{\ \mathrm{id},\ z\mapsto \mathrm{e}^{\ii\theta}/z,\ z\mapsto -z,\ z\mapsto -\mathrm{e}^{\ii\theta}/z\ \}. $$
Let $g_\theta:z\mapsto \mathrm{e}^{\ii\theta/2}z$. Then $\Gamma$ is conjugated to the standard $K4$ in $\PSUtwo$:
$$ g_\theta^{-1}\Gamma g_\theta = \{\ \mathrm{id},\ z\mapsto 1/z,\ z\mapsto -z,\ z\mapsto -1/z\ \} := K. $$
So we shall now assume that $\Gamma=K$ as above.

\medskip
\textbf{(Step 4).} Solving non-trivial $[\phi]\in G_\Gamma/\Otwo$.

Since $\Gamma$ also keeps $\{b/\bar{a}, -a/\bar{b}\} \neq\{0,\infty\}$, and the only fixed points of $\gamma_0:z\mapsto -z$ are $0,\infty$, we have $b/\bar{a}=a/\bar{b}$. Then $\abs{a}^2=\abs{b}^2=1/2$, and $a=1/\sqrt{2}$ by normalization. Let $b=\mathrm{e}^{\ii\beta}/\sqrt{2},\ 0\leqslant \beta < \pi$. Then $\{b/\bar{a}, -a/\bar{b}\}=\{ \pm \mathrm{e}^{\ii\beta} \}$ is invariant under $z\mapsto 1/z$.

a). $1/\mathrm{e}^{\ii\beta} = -\mathrm{e}^{\ii\beta}$. Then $\beta=\pi/2$. We have
$$ \phi_1 = \frac{z+\ii}{\ii z +1}. $$

b).$1/\mathrm{e}^{\ii\beta} = \mathrm{e}^{\ii\beta}$. Then $\beta=0$. We have
$$ \phi_2 = \frac{z+1}{-z+1}. $$

So if $G_\Gamma/\Otwo$ is non-trivial, $\Gamma\cong K4$ and there will be 2 non-trivial elements $[\phi_1], [\phi_2]\in G_\Gamma/\Otwo$ given above. The reverse automatically holds by Step (4).
\end{proof}

In $\SOthr$, $\phi_1, \phi_2$ are the $\pi/2$-rotation about the $x,y$-axis, and $\gamma_0$ is the $\pi$-rotation about the $z$-axis. We can restate the result in different languages. 
\begin{corollary}\label{cor:special_D} Let $K\subset\PSUtwo$ and $\phi_1, \phi_2 \in \PSUtwo$ be defined as above, and $\Sigma$ be a strict dihedral surface.
\begin{enumerate}
    \item  $K, \phi_1^{-1} K \phi_1, \phi_2^{-1} K \phi_2$ are three subgroups of $\Otwo$ which are conjugate in $\PSUtwo$ but not in $\Otwo$.
    \item Let $\Sigma_K := \CC\cup\{\infty\}/ K$. Geometrically, $\Sigma_K$ is a genus zero spherical surface with three cone points of angle $\pi$. Then $\#\{ G_\Gamma/\Otwo \}=3$ if and only if $\Sigma$ admits a branched cover to $\Sigma_K$. 
\end{enumerate}
\end{corollary}
\begin{proof}
    (1). This is a direct corollary from the proof above.

    (2). If $\#\{ G_\Gamma/\Otwo \}=3$, then $\rho(\pi_1)\cong K$. 
    By the $\rho$-equivariance of the developing map $D:\hat{\Sigma}\to\STwo$, $D$ reduces to a map from $\Sigma$ to $\Sigma_K$, mapping the cone points to cone points of $\Sigma_K$. It is a covering map outside the orbifold point of $\Sigma_K$ since $D$ is locally isometric on $\tilde{\Sigma}'$.
    Conversely, note that $\Sigma_K$ admits 3 different foliated surfaces. Each cone point of $\Sigma_K$ can be the pole of the foliation and the other two cone points are the zeros with index $+1$. If there is a branched cover $\hat{D}:\Sigma\to\Sigma_K$, the pullback foliations give the 3 different foliated surface of $\Sigma$.
\end{proof}

\subsection{Connection matrix}
There are as many linear equations on the widths of bigons as the number of poles. The connection matrix helps us detect their linearly independence. 
The idea is the same as \cite{Li19} used for co-axial surfaces.

Let $(\Sigma;F,G)$ be a generic foliated surface, with $m$ poles $\{q_1,\cdots, q_m\}$ of $F$. Assume its strip decomposition consists of $b$ bigons $\{B_1, \cdots, B_b\}$.
\newcommand{\mtxM}{\mathbf{M}}
\begin{definition}
\label{def:cnnct_mtr}
    The \DEF{connection matrix} \SYM{$\mtxM$} of $(\Sigma;F,G)$ is an $m \times b$ matrix with all entries in $\{0,1,2\}$, determined by the following:
    \begin{itemize}
        \item if both vertices of bigon $B_j$ are glued to one pole $q_i$, then $\mtxM_{ij}:=2$ ;
        \item if vertices of $B_j$ are glued to different poles $q_i\neq q_{i'}$, then $\mtxM_{ij}=\mtxM_{i'j}:=1$;
        \item all the other entries are $0$.
    \end{itemize}
\end{definition}

Since each bigon has exactly two vertices, each column of $\mtxM$ contains 1 or 2 non-zero entries, with total sum 2.
Let $\vec{C}:=(c_1, \cdots, c_m)^T$ be the column vector recording cone angles at poles $q_i$ on $\Sigma$, and $\vec{W}:=(w_1, \cdots, w_b)^T$ be the vector recording widths of $B_j$. Then the constrain on the widths of bigons is given by
\begin{equation}\label{eq:width_constrain}
    \vec{C} = \mtxM  \vec{W}.
\end{equation}

\newcommand{\rank}{\mathrm{rank}}
\newcommand{\rankM}{\mathrm{rank}(\mtxM)}
\begin{proposition}\label{prop:rank} For generic foliated surface $(\Sigma;F,G)$, 
\begin{enumerate}
    \item $\rankM \geq m-1$;
    \item the foliation pair $(F,G)$ is orientable if and only if $\rankM = m-1$.
\end{enumerate} \end{proposition}

The key to the proof of Proposition \ref{prop:rank} is the connectedness of $\mtxM$ observed in \cite{Li19}. We include a proof here for convenience.
\begin{lemma}
    For generic $(\Sigma;F,G)$, its connection matrix $\mtxM$ is path-connected in the following sense: for any two non-zero elements $\mtxM_{ij}, \mtxM_{kl}$, there exists a finite sequence $(i_0,j_0), (i_1,j_1), \cdots, (i_N, j_N)$ such that
    \begin{itemize}
        \item $(i_0,j_0)=(i,j), (i_N, j_N)=(k,l)$;
        \item $(i_{t-1},j_{t-1}), (i_{t},j_{t})$ have exactly one same component for all $t=1,\cdots,N$;
        \item all entries $\mtxM_{(i_{t},j_{t})}$ are non-zero.
    \end{itemize}
\end{lemma}

\begin{proof}
    This is related to the connectedness of the surface. Up to a re-labeling of the poles and bigons, we may assume $\mtxM_{11}, \mtxM_{kl}\neq 0$ and $k, l \neq 1$. 
    Choose a path $\sigma$ on $\Sigma$ from an interior point in $B_1$ to an interior point in $B_l$, avoiding all singularity of $F$. Then consider all the bigons, assumed to be $B_{j_1}=B_1, B_{j_2}, \cdots, B_{j_L}=B_l$, along the path $\sigma$. Whenever $\sigma$ goes into $B_{j_{i+1}}$ from $B_{j_i}$, these two bigons glued to a common pole because the foliated surface is generic and each boundary segment ends at a pole. Then there exists a desired sequence among the columns of $B_{j_1},\cdots, B_{j_L}$.
\end{proof}

\textbf{Remark.} The connection matrix in \cite{Li19} is always connected, while ours may not be connected for non-generic surfaces. These two matrices encode ``dual'' information: Li's is obtained from annulus decomposition while ours is induced by strip decomposition. 

\begin{proof}[Proof of Proposition \ref{prop:rank}. ] ~

    (1). We show that if $\rankM < m-1$, then $\mtxM$ is not path-connected.
    In that situation, any $m-1$ rows of $\mtxM$ are linearly dependent. Up to re-labeling, we may assume that there exists non-zero $\lambda_i$ for $i=1,\cdots , r$ such that
    $r\leq m-1$ and $\sum_{i=1}^{r} \lambda_i \vec{M}_i = 0$,
    where $\vec{M}_i$ is the $i$-th row of $\mtxM$.

    Now pick any non-zero entry $\mtxM_{xy}$ on the last $m-r$ row (i.e., $r< x \leq m$). Such entry exists because any pole has positive cone angle. Note that the total sum of each column is 2. If $\mtxM_{xy}=2$, then $\mtxM_{iy}=0$ for all $i\neq x$. If $\mtxM_{xy}=1$, then there exists a uniqe $x'\neq x$ with $\mtxM_{x' y}=1$, and $\mtxM_{iy}=0$ for all $i\neq x, x'$.
    However, the first $r$ rows are linearly dependent so $\sum_{i=1}^{r} \lambda_i \mtxM_{iy} =0$, with all $\lambda_i\neq0$. Thus $x'>r$.

    So for all $j=1,\cdots, b$, if $1\leq i \leq r$ and $r+1 \leq i' \leq m$, them $\mtxM_{ij}, \mtxM_{i'j}$ can not be simultaneously non-zero. This means the entries in the first $r$ rows and ones in the last $m-r$ rows are not path-connected, a contradiction.

    \bigskip
    (2). We show that the orientation induces a linear equation. Note that as a pair of topological foliation induced by some meromorphic differentials, $F,G$ must be orientable or non-orientable simultaneously.

    When $(F,G)$ is orientable, the two vertices of each bigon glues to different poles, and all non-zero entries of $\mtxM$ is $1$. Pick one orientation of $F$. Then for each $B_j$, we can tell which vertex lies on the left or right of $F$, and it is well-defined all over the surface. Define $\lambda_i=+1$ if $q_i$ lies on the left of $F$, otherwise define $\lambda_i=-1$. Then 
    \[ \sum_{i=1}^m \lambda_i \vec{M}_i =0, \] 
    since each bigon is glued to exactly one pole on the left and one pole on the right. So $\mtxM$ is not full-rank, then $\rankM$ must be $m-1$.

    \medskip
    Conversely, if $\rankM= m-1$, then there exists non-zero $\lambda_i$ such that $\sum_{i=1}^m \lambda_i \vec{M}_i =0$. As before, no entry in $\mtxM$ equals to $2$, otherwise it is the only non-zero entry on that column and the coefficient $\lambda_i$ for that row must be 0. Then for each column there is exactly two non-zero entry, both equal to 1. The coefficients $\lambda_i$ for that two rows must be opposite each other. By the connectedness of $\mtxM$, all $\lambda_i$ have the same absolute value. Then we shall assume that $\lambda_i=\pm1$.

    An orientation of $F$ can be given as the following. On each bigon, orient the restriction of $F$ so that the pole with positive $\lambda_i$ lies on the left. This orientation matches piece by piece, since all bigons glued to the same poles are recorded in one row, with the same coefficient. And the linear equation guarantees that the two poles to which a bigon is glued lie on different side of $F$.
\end{proof}

\begin{example}
    The connection matrices, together with the equations (\ref{eq:width_constrain}), for the two strip decomposition in Example \ref{eg:three_cone_sphere} are
    \[
    \begin{pmatrix}    1/2  \\ 1/2  \\ 1 \\ 1    \end{pmatrix}
    = \begin{pmatrix}
       1 &   &   \\
         &   & 1 \\
       1 & 1 &   \\
         & 1 & 1 \\
    \end{pmatrix}
    \begin{pmatrix}    1/2  \\ 1/2  \\ 1/2    \end{pmatrix}
    \quad \text{ and } \quad
    \begin{pmatrix}    1  \\ 1  \\ 1     \end{pmatrix}
    = \begin{pmatrix}
       2 & 1 &   &  \\
         & 1 & 1 &  \\
         &   & 1 & 2\\
    \end{pmatrix}
    \begin{pmatrix}    \theta  \\ 1-2\theta  \\ 2\theta  \\ 1/2-\theta  \end{pmatrix}.
    \]
    ~ \hfill $\square$
\end{example}

\subsection{Dimension counting for strict dihedral case}
Now we are ready to count the dimension of $\MsphgnDT$.

\begin{theorem}\label{thm:dim_count_DT}
    Let $\vec{\theta}:=(\theta_1, \cdots, \theta_n) \in \RR^n_+$ be an angle vector with $\theta_i\neq1$ and $\vec{T}=( E, O, N )$ be a type partition with $ \#E=n_E, \#O=n_O, \#N=n_N $.
    If $\MsphgnDT$ is non-empty, then its real dimension is 
    $$ 2n_E+n_O+2g-2. $$
\end{theorem}
\begin{proof}
    We first prove the statement for $\MsphgnDFT$, then project it to $\MsphgnDT$.

    \textbf{(Step 1).} By Corollary \ref{cor:low_dim}, to count the real parameters for $\MsphgnDFT$, we only need to consider generic foliated surfaces. 
    Let $(\Sigma;F,G)\in \MsphgnDFT$ be a generic foliated surface. Suppose that there are $b$ bigons in the strip decomposition, and $m\geq0$ extra poles other than the cone points.


    The strip decomposition can be viewed as a cell decomposition of $\Sigma$, and each bigon is a topological quadrilateral. By computing the Euler characteristic we have
    \[ 2-2g = (n_E + n_O + n_N + m) - 2b + b. \]
    So $b = m+n+(2g-2) $. 
    Every zero of even index provides one parameter, namely its absolute latitude. Every zero of odd index must lie on the equatorial net (Definition \ref{def:singular_info}), not providing new parameter.

    The widths of the $b$ bigons are the other real parameters, constrained by $n_N+m$ linear equations $\vec{C}=\mtxM \vec{W}$, where $\mtxM$ is the $(n_N+m)\times b$ connection matrix. By Proposition \ref{prop:rank}, for strict dihedral surface, $\mtxM$ is full-rank. Then the dimension of the solution space of the width vector $\vec{W}$ is
    \[ b - \rankM = \big(m+n+(2g-2)\big) - \big( n_N+m \big) = n - n_N + 2g-2. \]
    Together with the parameters from the zeros of even index, the total dimension is
    \[ n_E + (n - n_N + 2g-2) = 2 n_E + n_O + 2g-2. \]

    \textbf{(Step 2).} By Proposition \ref{prop:id_SDhd}, the fibers of $p_T:\MsphgnDFT \to \MsphgnDT$ is always finite, so the projection $p_T$ preserves dimension. Since generic foliated surfaces form the top-dimensional subset of $\MsphgnDFT$ by Corollary \ref{cor:low_dim}, the real dimension of $\MsphgnDT$ is the same as $\MsphgnDFT$.    
\end{proof}

The maximal dimension of $\MsphgnD$ is defined as the maximal dimension of $\MsphgnDT$ among all possible type partition $\vec{T}$.

\begin{definition}
    Given an angle vector $\vec{\theta}$, the \DEF{maximal type partition} $\vec{T}_0$ is defined to be the type partition $\{E_0,O_0,N_0\}$ such that
    \[ T := \sum_{i\in E_0\sqcup O_0} \theta_i \] 
    is integer and maximal.
    Denote $\#E_0= M_E ,\#O_0=M_O, \#N_0=M_N$. Then $M_O$ must be even. We use upper case $M$ for maximal type partition and lower case $n$ for general type partition.
\end{definition}
For the maximal type partition, all indexes with integer angle is in $E_0$. If the total number of half-integers is even, all their indexes belong to $O_0$. If the number is odd, $O_0$ contains all but the index of the smallest half-integer.
$T$ is called the \DEF{maximal integral sum} in \cite{GT23}. A necessary condition for the existence of dihedral surface, called \DEF{strengthened Gauss-Bonnet inequalities}, is given in terms of $T$ and $g,n$ there.

\begin{theorem}\label{thm:dim_count_D}
    Let $\vec{\theta}:=(\theta_1, \cdots, \theta_n) \in \RR^n_+$ be an angle vector with $\theta_i\neq1$, and let $\vec{T}_0$ be the maximal type partition as above. 
    Then whenever the moduli space is non-empty, the maximal dimension of $\MsphgnD$ is
    \[ 2M_E+M_O+2g-2. \]
\end{theorem}

\begin{proof}
    In \cite[Section 5.1, Section 6.2]{GT23}, it is shown case by case that whenever $\MsphgnD$ is non-empty, one can find a strict dihedral surface with maximal type partition $\vec{T}_0$.
    So $\MsphDT{g}{n}{\vec{\theta}}{\vec{T}_0}$ is always non-empty, with the desired dimension.
\end{proof}

\subsection{Dimension counting for co-axial case}
Similar to the strict dihedral case, we start from $\MsphgnCFT$.
\begin{lemma}\label{lem:dim_count_CFT}
    Let $\vec{\theta}:=(\theta_1, \cdots, \theta_n) \in \RR^n_+$ be an angle vector with $\theta_i\neq1$ and $\vec{T}=( E, \emptyset, N )$ be a type partition with $ \#E=n_E, \#N=n_N $.
    If $\MsphgnCFT$ is non-empty, then its real dimension is 
    $ 2n_E+2g-1 $.
\end{lemma}
\begin{proof}
    This is almost identical to (Step 1) of Theorem \ref{thm:dim_count_DT}. By Corollary \ref{cor:low_dim} we only need to consider generic foliated surfaces. Suppose there are $b$ bigons in the strip decomposition of a generic surface, and there are $m$ poles in total, then $2-2g=(n_E+m)-2b+b$ and $b=n_E+m+2g-2$.
    By Proposition \ref{prop:rank}, the rank of connection matrix here is $m-1$. So the dimension of the solution space of width vector is $b-\rankM = (n_E+m+2g-2)-(m-1) = n_E+2g-1$. Together with the parameter provided by the $n_E$ zeros of even index, the total dimension is $n_E+(n_E+2g-1)$.
\end{proof}

The main difficulty of co-axial case lies in the study of vertical arrows in (\ref{eq:moduli_spaces_C}).
A developing map is determined by its behavior on an open set. Thus whenever $N\neq \emptyset$ in the type partition $\vec{T}$, the projection $p_T$ is injective. So we only need to consider the case where $N=\emptyset$. For co-axial surfaces, this implies $\theta_i\in\NN_{>1}$ for all $i$. 

\begin{lemma}[Identification lemma for co-axial surfaces]\label{prop:id_CoAx} ~ \\
Suppose $N=\emptyset$ in the type partition.
    \begin{enumerate}
    \item If the monodromy group of $\Sigma$ is trivial, then $p_T^{-1}(\Sigma)$ is a two-dimensional fiber in $\MsphgnCFT$.
    \item If the monodromy group of $\Sigma$ is not trivial, then the foliated surface in $\MsphgnCFT$ with orientable $(F,G)$ is unique.
    \end{enumerate}
\end{lemma}
\begin{proof}
    (1). When the monodromy is trivial, the developing map induces a branched cover $D$ directly from $\Sigma$ to the Riemann sphere. Then composing any $\PSUtwo$ element on the left still pulls back the foliation. To fix type partition, the preimages of two poles must be disjoint from the cone points. 
    The choices of such elements form an open subset of $\PSUtwo$, and its quotient by $\Otwo$ is 2-dimensional.

    \medskip
    (2). We follow notations in Proposition \ref{prop:id_SDhd}, but $\Gamma=\Gamma_0\subset\Uone$ this time. By (Step 2) and (Step 3) of its proof, $G_\Gamma/\Otwo$ must be trivial if $\Gamma$ is not $\ZZ_2$ or the trivial group.

    If $\Gamma\cong \ZZ_2$, then up to a conjugation we may assume $\Gamma=\{\mathrm{id},\ \gamma_0:z\mapsto -z\}$. Let $\phi\in\PSUtwo$ as before, with $ab\neq0$. Then
    $$ \phi^{-1}\gamma_0\phi: z \mapsto \frac {(\abs{a}^2-\abs{b}^2)z + 2\bar{a}b} {2a\bar{b}z + (\abs{b}^2-\abs{a}^2)}. $$
    This is in $\Otwo$ if and only if $\bar{a}b(\abs{a}^2-\abs{b}^2)=0$. Thus $a=1/\sqrt{2}$ and $b=\mathrm{e}^{\ii\beta}/\sqrt{2},\ 0\leqslant \beta < \pi$ as before. One can check that
    $$ \phi=\frac{z+\mathrm{e}^{\ii\beta}}{-\mathrm{e}^{-\ii\beta}z+1},\quad  \phi^{-1}\gamma_0\phi=\frac{\mathrm{e}^{2\ii\beta}}{z}$$
    satisfies the requirements for all $\beta\in[0,\pi)$. 
    However, this new element lies in 
    $\Otwo-\Uone$, so the induced foliation is non-orientable. This is forbidden in $\MsphgnCFT$ by Definition \ref{def:moduli_C}.(2). So the choice of orientable $(F,G)$ is actually unique.
\end{proof}

As a corollary, if the monodromy group of $\Sigma$ is $\ZZ_2$, then this \csphmtr\ is induced by both a unique 1-form and a 1-parameter family of primitive \qdf s. This is exactly the case of Example \ref{eg:three_cone_sphere}. Combining this with Corollary \ref{cor:special_D}.(2), we can summary as the following.
\begin{corollary}
    Let $\Sigma$ be a dihedral surface. The following conditions are equivalent:\\
    \begin{itemize}
        \item $\Sigma$ admits multiple pairs of latitude and longitude foliations, or multiple strip decompositions;
        \item $\Sigma$ has monodromy group isomorphic to a subgroup of $K4$;
        \item $\Sigma$ admits a branched cover of $\Sigma_K$ defined in Corollary \ref{cor:special_D}.
    \end{itemize}
\end{corollary}

Trivial monodromy is the only obstruction of injective projection. However, the occurrence of trivial monodromy is infrequent and varies depending on the genus.
\begin{proposition}\label{prop:trivial_mono}
    Suppose $\vec{\theta}\in\ZZ_{>1}^n$ and $N=\emptyset$ in the type partition $\vec{T}$.
    \begin{enumerate}

        \item If $g=0$, then any $\Sigma\in \MsphC{0}{n}{\vec{\theta}}$ has trivial monodromy.
        \item If $g>0$, then the surfaces with trivial monodromy form a lower dimensional subset in $\MsphgnCT$.
    \end{enumerate}
\end{proposition}
\begin{proof}
    (1). When $g=0$, the simple primitive closed loops around each cone point generate $\pi_1(\Sigma-P)$. Thus if all $\theta_i$ are integers, the monodromy group is trivial.

    \medskip
    (2). We show that generic foliated surfaces with trivial monodromy form a lower dimensional subset in $\MsphgnCFT$, and project the result to $\MsphgnCT$.

    \textbf{(Step 1).} Some homology preparations are needed.

    Let $(\Sigma;F,G)\in\MsphgnCFT$ be a generic foliated surface, and fix an orientation of $F$. Suppose there are $m>0$ poles of $F$ in total, and $b$ bigons in the strip decomposition. For each bigon $B_j$ in the strip decomposition, pick an interior segment $e_j$ connecting the two boundary marked points, and orient it along $F$. Each $B_j$ is cut into two half-pieces by $e_j$, and all half-pieces glued to the same pole $q_l$ form a topological disk $d_l$. 
    The zeros of $F$, the segments $\{ e_j \}_{j=1}^{b}$, and the disks $\{d_l\}_{l=1}^{m}$ form a cell decomposition of $\Sigma$.

    Since $\Uone$ is abelian, $\mathrm{Hom}(\pi_1(\Sigma-P), \Uone) = \mathrm{Hom}(H_1(\Sigma-P), \Uone)$. When all $\theta_i\in\ZZ$, the closed loops around every cone point have trivial monodromy. So we can further replace $H_1(\Sigma-P)$ with $H_1(\Sigma)$. Then the oriented segments $\{e_j\}$ form a basis of the chain group $C_1(\Sigma)$. The effect on the monodromy of moving along $e_j$ is a $2\pi w_j$-rotation.
    Let $a_1, \cdots , a_{2g} $ be a canonical basis of $H_1(\Sigma)$. Then as closed 1-chains, there exists a $2g\times b$ integer matrix $X$ such that
    \[\begin{pmatrix} a_1  \\ \vdots \\ a_{2g}  \end{pmatrix} =
    X \begin{pmatrix} e_1  \\ \vdots  \\ e_b \end{pmatrix} \]
    with $\rank(X)=2g$. Suppose $\rho$ is the monodromy representation of $(\Sigma;F,G)$. For $\gamma\in H_1(\Sigma)$, if $\rho(\gamma)=z \mapsto e^{2\pi\ii \psi} \in \Uone$, define $(\ln\rho)(\gamma):=\psi\in[0,1)$. Then we have
    \[\ln\rho \begin{pmatrix} a_1  \\ \vdots \\ a_{2g}  \end{pmatrix} \equiv
    X \begin{pmatrix} w_1  \\ \vdots  \\ w_b \end{pmatrix} \quad (\text{mod } \ZZ). \]

    On the other hand, since $F$ is orientable, the connection matrix contains no entry 2, and we can improve it with signature. As in the proof of Proposition \ref{prop:rank}.(2), define $\mtxM_{ij}=+1$ if $B_j$ is glued to pole $q_i$ that lies on the left of $F$, and define $\mtxM_{ij}=-1$ if $q_i$ lies on the right of $F$. Then the total sum of each column in this $\mtxM$ is 0.
    The boundary chain of each disk $d_l$ consists of segments with same signature. So we can orient them properly so that
    \[\begin{pmatrix} \pt d_1  \\ \vdots \\ \pt d_m  \end{pmatrix} =
    \mtxM \begin{pmatrix} e_1  \\ \vdots  \\ e_b \end{pmatrix}, \]
    where $\pt:C_2(\Sigma) \to C_1(\Sigma)$ is the boundary operator.

    Since $a_1,\cdots,a_{2g}$ generates $H_1(\Sigma)$, they must be linearly independent at chain level. Then rows of $X$ are linearly independent. However, each row of $\mtxM$ gives a boundary chain. So each row of $X$ is linearly independent of all rows in $\mtxM$.

    \medskip
    \textbf{(Step 2).} Suppose $(\Sigma;F,G)$ has trivial monodromy and $\vec{W}\in\RR_{>0}^b$ is the width vector. Then there exists $\vec{C}_0\in\ZZ_+^m$ and $\vec{A}\in\ZZ^{2g}$ such that
    \begin{equation}\label{eq:width_constrain_mond}
        \vec{C}_0=\mtxM \vec{W},\quad \vec{A}=X \vec{W}.
    \end{equation}
    We show that 
    the dimension of solution space is strictly smaller than $(b-\rankM)$. 
    
    By Proposition \ref{prop:rank}, $\rankM = m-1$. Let $\boldsymbol{M}_1,\cdots,\boldsymbol{M}_{m-1}$ be $(m-1)$ linearly independent rows of $\mtxM$, and $\boldsymbol{X}_1,\cdots,\boldsymbol{X}_{2g}$ be the rows of $X$.
    If these rows are linearly dependent, there exists $\lambda_1,\cdots,\lambda_{m-1},\mu_1,\cdots,\mu_{2g}$ such that
    $$ \lambda_1 \boldsymbol{M}_1 + \cdots + \lambda_{m-1} \boldsymbol{M}_{m-1} = \mu_1 \boldsymbol{X}_1 + \cdots + \mu_{2g} \boldsymbol{X}_{2g}. $$
    and at least one $\lambda_i$ and $\mu_j$ is non-zero.
    Applying this equation on the basis $\left( e_1, \cdots, e_b \right)^T$, the left side becomes a linear combination of $\pt d_1,\cdots,\pt d_m$, which must be exact. And the right side becomes a linear combination of the homology basis $a_1,\cdots,a_{2g}$, which can not be exact. Then both sides are zero. Since $\boldsymbol{X}_j$'s are linearly independent, all $\mu_j$ is zero, a contradiction.
    Thus
    $$\mathrm{rank}\begin{pmatrix} \mtxM \\ X \end{pmatrix} \geq m-1+2g$$
    and the solution space of (\ref{eq:width_constrain_mond}) has dimension $\leq b-(m-1+2g) < b-(m-1)$.
    The choices for $\vec{C}_0\in\ZZ_+^m$ and $\vec{A}\in\ZZ^{2g}$ are finite because $0<w_i\leq\max\{\theta_1,\cdots,\theta_n, 2\pi\}$. So the solution space for width vector is a finite union of subsets with dimension strictly smaller than $b-(m-1)= b-\rankM$. 
    So generic foliated surfaces with trivial monodromy form a lower dimensional subset in $\MsphgnCFT$

    \medskip
    \textbf{(Step 3).} By Lemma \ref{prop:id_CoAx}.(2), $p_T$ is injective on the subset of generic foliated surfaces with non-trivial monodromy. We have shown that in $\MsphgnCFT$, generic foliated surfaces with trivial monodromy form a lower dimensional subset, and non-generic foliated surfaces already form a lower dimensional subset. Since $p_T$ does not increase the dimension, the image of generic surfaces with non-trivial monodromy must form the top dimensional subset in $\MsphgnCT$.
\end{proof}

Finally, we are able to state the dimension count for co-axial surfaces.

\begin{theorem}\label{thm:dim_count_CT}
    Let $\vec{\theta}=(\theta_1, \cdots, \theta_n) \in \RR^n_+$ be an angle vector with $\theta_i\neq1$ and $\vec{T}=( E,\emptyset, N )$ be a type partition with $ \#E=n_E, \#N=n_N $.
    Whenever $\MsphgnCT$ is non-empty, its real dimension is given by the following:
    \begin{itemize}
        \item If $g>0$, the dimension is $2n_E+2g-1$;
        \item If $g=0$ and $N$ non-empty, the dimension is $2n_E-1$;
        \item If $g=0$ and $N$ empty, then $\vec{\theta}\in\NN_{>1}^n$ and the dimension is $2n-3$.
    \end{itemize}
\end{theorem}
\begin{proof}
    (1). If $g>0$, according to Proposition \ref{prop:trivial_mono}, the set of generic foliated surfaces with non-trivial monodromy has the top dimension in $\MsphgnCFT$, and it projects injectively onto a top dimensional subset in $\MsphgnCT$.
    Therefore, both moduli spaces have the same dimension, which is $2n_E+2g-1$ from Lemma \ref{lem:dim_count_CFT}. 

    (2). If $g=0$ while $N\neq\emptyset$, the projection $p_T$ is automatically injective. As a result, both moduli spaces once again have the same dimension. 

    (3). If $g=0$ and $N=\emptyset$, Proposition \ref{prop:trivial_mono}.(1) and Lemma \ref{prop:id_CoAx}.(1) imply that $\MsphCT{0}{n}{\vec{\theta}}{\vec{T}}$ is 2 dimensional lower than $\MsphCFT{0}{n}{\vec{\theta}}{\vec{T}}$. Therefore, the real dimension of the former is
    \[ (2n_E + 2g -1) -2 = 2n-3.\]
\end{proof}

We can also consider the maximal dimension of $\MsphgnC$. This is simpler than the strict dihedral case.
\begin{theorem}\label{thm:dim_count_C}
    Let $\vec{\theta}:=(\theta_1, \cdots, \theta_n) \in \RR^n_+$ be an angle vector with $\theta_i\neq1$. Suppose there are $M_E$ integers and $M_N$ non-integers, with $M_E + M_N = n$. Then whenever the moduli space is non-empty, the dimension of $\MsphgnC$ is
    \[\left\{   \begin{array}{ll}
        2M_E+2g-1 & , \text{ if } (g, M_N)\neq(0,0) ; \\
        2n-3 & , \text{ if } (g, M_N)=(0,0) .\\
                \end{array}
    \right. \]
\end{theorem}
\begin{proof}

    Let $\vec{T}_0=(E_0,\emptyset,N_0)$ be the type partition with $i\in E_0$ for all $\theta_i\in\ZZ$. For any type partition $\vec{T}=(E,\emptyset,N)$ other than $\vec{T}_0$, $N$ contains at least one index $i$ with $\theta_i\in\ZZ$. By Proposition \ref{prop:low_dim_bd}, $\MsphCF{g}{n}{\vec{\theta};\vec{T}}$ appears at the lower-dimensional boundary of $\MsphCF{g}{n}{\vec{\theta};\vec{T}_0}$.
    When $(g,N_N)\neq(0,0)$, this result projects to $\MsphC{g}{n}{\vec{\theta};\vec{T}_0}$ since $p_T$ do not increase dimension and is injective on some top-dimensional subset.
    When $(g,N_N)=(0,0)$, every fiber of $p_T$ is 2-dimensional. So the difference of dimension between subsets are kept by this projection.
    Thus, the maximal dimension of $\MsphgnC$ is always achieved by $\MsphC{g}{n}{\vec{\theta};\vec{T}_0}$.
\end{proof}

\textbf{Remark.} Li \cite{Li19} has proved that the real dimension of moduli space of Strebel 1-form is $2g+l-1$, where $l=n_E$ is the number of zeros. The extra $n_E$ dimensions just arise from the sliding deformation of the zeros along longitude foliation.

\subsection{Examples for distinct type partitions}
The final part contains two examples of moduli space with the same prescribed cone angle but different type partition.

\begin{example}\label{eg:intersect_comp}
    Let $g=1, n=3, \vec{\theta}=(4,\frac12,\frac32)$. Consider two type partitions
    $$\vec{T}_A:=\left(\{4\};\{\frac12,\frac32\};\emptyset \right) \quad \text{ and }\quad 
    \vec{T}_B:=\left(\{4\};\emptyset;\{\frac12,\frac32\}\right). $$
    Then $\MsphD{1}{3}{\vec{\theta};\vec{T}_A}$ and $\MsphD{1}{3}{\vec{\theta};\vec{T}_B}$ has a one-dimensional intersection. 
\end{example}
\begin{proof}
By Theorem \ref{thm:dim_count_DT},
$$\dim_\RR \MsphD{1}{3}{\vec{\theta};\vec{T}_A}=4,\quad
    \dim_\RR \MsphD{1}{3}{\vec{\theta};\vec{T}_B}=2. $$
They are not empty once we construct their intersection. If $\Sigma$ lies in their intersection, it admits at least two different choices of $(F,G)$. By Proposition \ref{prop:id_SDhd}, its monodromy group must be $K4$. Similar to the proof of Proposition \ref{prop:trivial_mono}.(2), we can show that the surfaces with $K4$-monodromy also form a lower dimensional subset. So the intersection of $\MsphD{1}{3}{\vec{\theta};\vec{T}_A}$ and $\MsphD{1}{3}{\vec{\theta};\vec{T}_B}$ is at most one dimensional.

On the other hand, we can precisely found a one-dimensional subset as Figure \ref{fig:example_intersection}. Here $p_1,p_2,p_3$ is the cone point of angle $8\pi,\pi,3\pi$ respectively. We are not considering the behavior of different components, but at least these two subspace already have a one-dimensional intersection, while their own dimension are different.

\begin{figure}[ht]
  \makebox[\textwidth][c]{\centering
  \includegraphics[width=\textwidth]{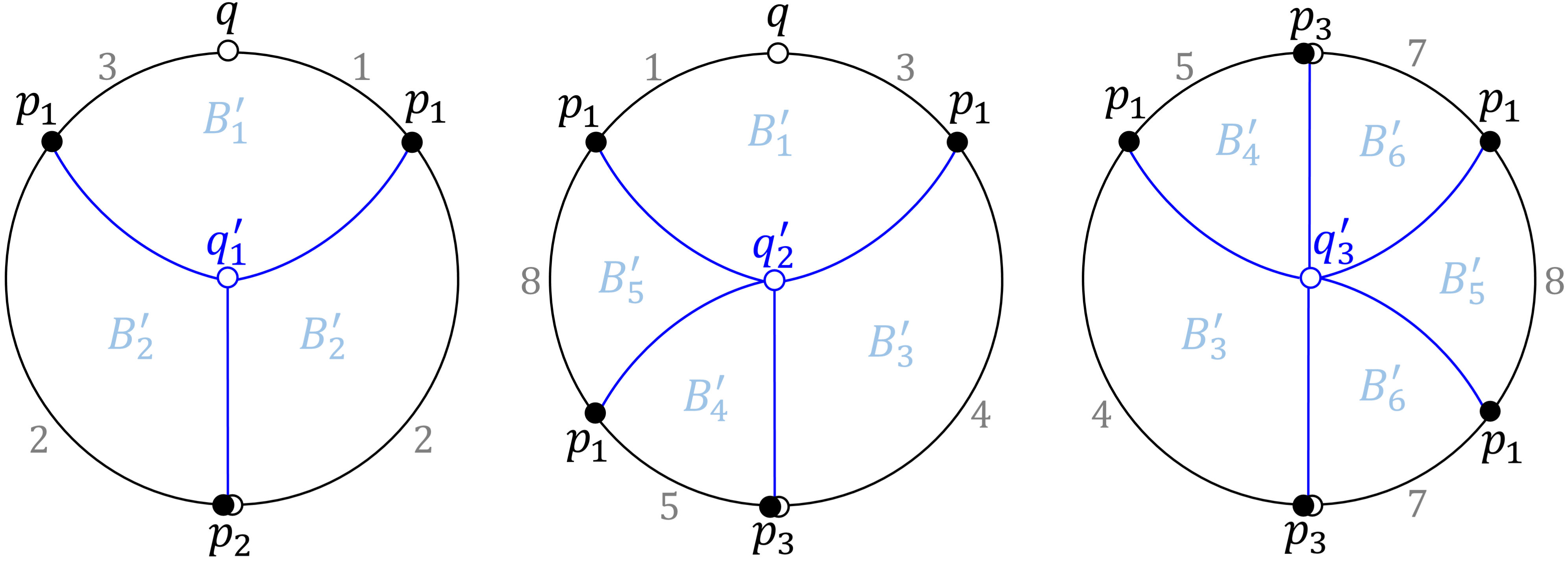}}
  \caption{A target surface is composed of 3 hemispheres. The boundary segments labeled with the same number are glued in pair. The one-parameter is the absolute latitude of $p_1$. The corresponding foliation contains three poles, including $p_2,p_3$ and an extra smooth point $q$. Its type partition is $\vec{T}_B$. 
  Blue segments result in another strip decomposition consisting of 6 bigons. All cone points are the zeros of the corresponding foliation. All poles of the new foliation are the center of the three hemispheres. Its type partition is $\vec{T}_A$.
  }
  \label{fig:example_intersection}
\end{figure}
\end{proof}

\begin{example}\label{eg:disjoint_comp}
    Now let $g=1, n=4, \vec{\theta}'=(4,\frac12,\frac32,c)$, where $c$ is an irrational number with $\frac12<c<2$. Consider type partitions $$\vec{T}'_A:=(\{4\};\{\frac12,\frac32\};\{c\}) \quad \text{ and }\quad \vec{T}'_B:=(\{4\};\emptyset;\{\frac12,\frac32,c\}). $$
    In this case $\MsphD{1}{4}{\vec{\theta}';\vec{T}'_A}$ and $\MsphD{1}{4}{\vec{\theta}';\vec{T}'_B}$ are disjoint in $\MsphD{1}{4}{\vec{\theta}'}$.
\end{example}
\begin{proof}
Note that $\vec{T}'_A$ is the maximal type partition, so $\MsphD{1}{4}{\vec{\theta}';\vec{T}'_A}$ is non-empty. A family of surfaces in $\MsphD{1}{4}{\vec{\theta}';\vec{T}'_B}$ are constructed in Figure \ref{fig:example_no_intersection}. Since $c$ is irrational, the choice of $(F,G)$ is unique for every surface in the mentioned moduli space. So foliated surfaces of different type partition must project to different strict dihedral surfaces.

\begin{figure}[ht]
  \makebox[\textwidth][c]{\centering
  \includegraphics[width=\textwidth]{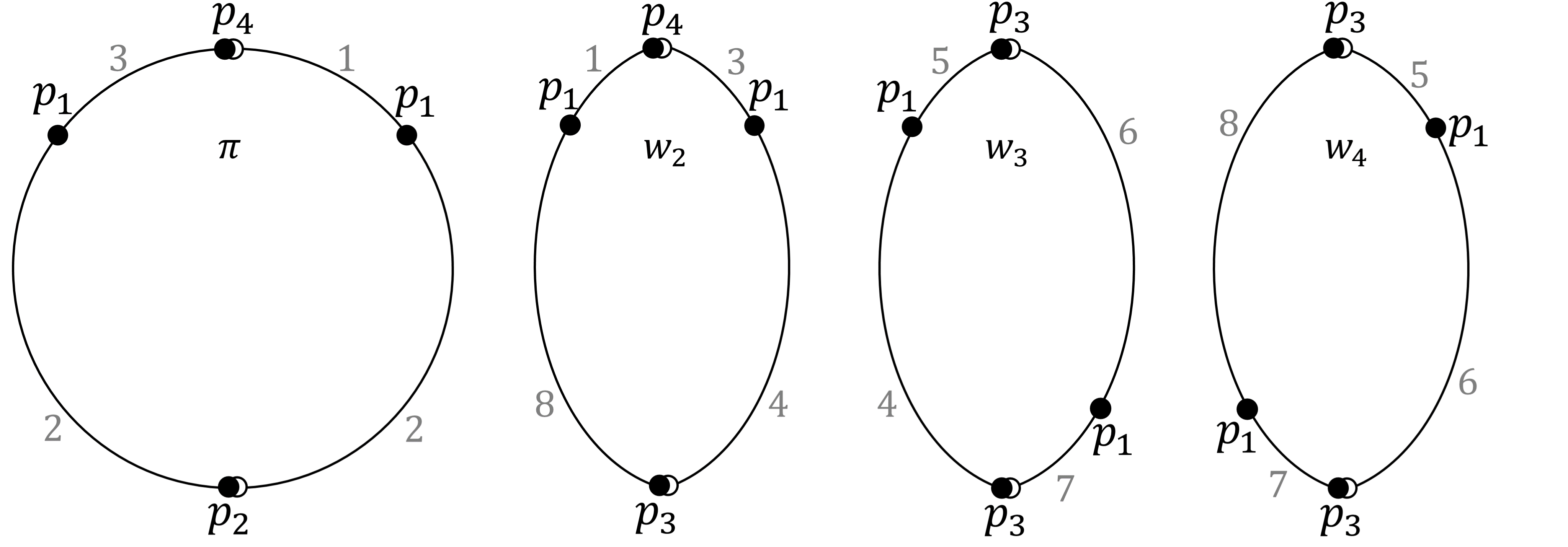}}
  \caption{The strip decomposition of a family of surfaces in $\MsphD{1}{4}{\vec{\theta}';\vec{T}'_B}$, consisting of 4 bigons of width $\pi,w_2,w_3,w_4$, where $w_2+1/2=c$ and $w_2+2(w_3+w_4)=\frac32$. The cone points $p_2, p_3, p_4$ are the poles of the corresponding foliation, with angle $\pi, 3\pi, 2\pi c$. } 
  \label{fig:example_no_intersection}
\end{figure}
\end{proof}

\bigskip
These examples imply the connected component worth serious study. This also provides a reason for studying several strata of meromorphic \qdf s simultaneously, and their relative position when projected to the moduli space.

\bigskip
We close with a summary to illustrate the flow chart at the very end. A dihedral cone spherical structure is induced by some special meromorphic abelian or quadratic differentials on the underlying Riemann surface. This results in a pair of latitude and longitude foliations. However, the relationships between these objects are complex. 

One differential can correspond to a family of dihedral surfaces, as spherical structures are finer than conformal structures. Yet multiple pairs of measured foliations can represent the same dihedral surface. Thus, the moduli space of dihedral cone spherical surfaces lies between the space of those special differentials and the space of measured foliation pairs. This fact exhibits both the excitement and challenges associated with cone spherical surfaces, even within this specific class.


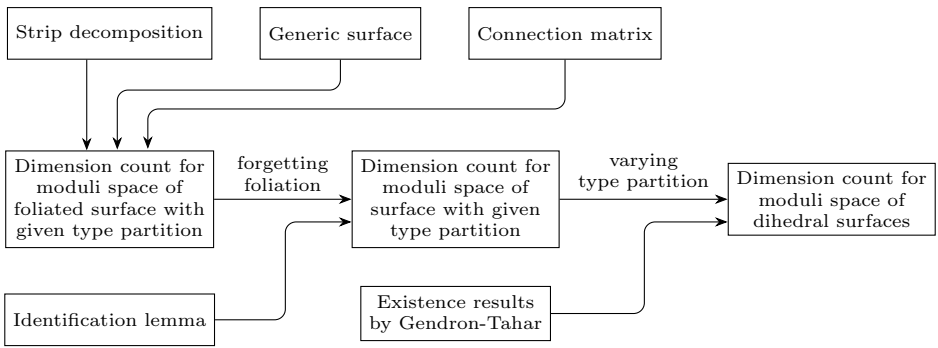
\begin{figure}[!h]
\makebox[\textwidth][c]{\centering
\resizebox{1.05\textwidth}{!}{
  \begin{tikzpicture}[node distance=2cm, >=Stealth]
    \tikzset{
      block/.style = {rectangle, draw, align=center, minimum height=2.1em, , font=\footnotesize},
    }
    \node [block] (strip) {Strip decomposition};
    \node [block, right=0.618cm of strip] (generic) {Generic surface};
    \node [block, right=0.618cm of generic] (matrix) {Connection matrix};
    \node [block, below=1.2cm of strip] (MsphFT) {Dimension count for\\ moduli space of \\foliated surface with \\given type partition};
    \node [block, below=0.6cm of MsphFT] (idlemma) {Identification lemma};
    \node [block, right=1.8cm of MsphFT] (MsphT) {Dimension count for\\ moduli space of \\surface with given \\type partition};
    \node [block, below=0.5cm of MsphT] (GenTah) {Existence results\\by Gendron-Tahar};
    \node [block, right=2.2cm of MsphT] (Msph) {Dimension count for\\moduli space of\\dihedral surfaces};
    \draw [->] ($(strip.south)+ (-0.3cm,0)$) -- ($(MsphFT.north)+ (-0.3cm,0)$);
    \draw [->][rounded corners] (generic) |- (1,-0.75) -| ($(MsphFT.north) + (0.1cm,0.01cm)$);
    \draw [->][rounded corners] (matrix) |- (3,-1) -| ($(MsphFT.north) + (0.5cm,0.02cm)$);
    \draw [->] (MsphFT) --node[align=center, above, font=\footnotesize]{forgetting\\foliation} (MsphT);
    \draw [->][rounded corners] (idlemma.east) -| (2.3,-3) |- ($(MsphT.west) + (0cm,-0.3cm)$);
    \draw [->] (MsphT) --node[align=center, above, font=\footnotesize]{varying\\type partition} (Msph);
    \draw [->][rounded corners] (GenTah.east) -| (7,-3) |- ($(Msph.west) + (0cm,-0.3cm)$);
  \end{tikzpicture}
}}
  \caption{The overall outline.} 
  \label{Summary}
\end{figure}




\phantomsection
\addcontentsline{toc}{section}{References}

\bibliography{sn-bibliography}




\bigskip

\fontsize{7}{9}\selectfont{
\Letter\ Sicheng Lu 

\textsc{ The Institute of Geometry and Physics, University of Science and Technology of China, Hefei 230026 China
}

\ \textit{Email address}: \texttt{sichenglu@ustc.edu.cn}

\bigskip
\ Bin Xu

\textsc{
School of Mathematical Sciences and CAS Wu Wen-Tsun Key Laboratory of Mathematics, University of Science and Technology of China, Hefei 230026 China
}

\ \textit{Email address}: \texttt{bxu@ustc.edu.cn}
}

\end{document}